\documentclass[aihp,preprint]{imsart}
\usepackage{amsmath,amsthm,amsfonts,amssymb,mathrsfs,bm}
\RequirePackage[OT1]{fontenc}
\usepackage[colorlinks=true]{hyperref}
\usepackage{ulem}
\usepackage{textcomp}
\usepackage{url}
\PassOptionsToPackage{hyphens}{url}\usepackage{hyperref}

\newtheorem{theorem}{Theorem}

\newtheorem{lemma}{Lemma}

\newtheorem{corollary}{Corollary}
\newtheorem{Prop}{Proposition}
\newtheorem{definition}{Definition}
\newtheorem{remark}{Remark}
\theoremstyle{remark}

\newtheorem{example}{Example}

 \def\bdefinition{\begin{definition}\sl{}\def\edefinition{\end{definition}}}

 \def\beqlb{\begin{eqnarray}}\def\eeqlb{\end{eqnarray}}
 \def\beqnn{\begin{eqnarray*}}\def\eeqnn{\end{eqnarray*}}

\def\P{\mathbb{P}}
\newcommand{\ddr}{\mathrm{d}}

\begin{document}
\begin{frontmatter}
\title{Continuous-state branching processes, extremal processes and super-individuals}
\runtitle{CSBPs, extremal processes and super-individuals}
\begin{aug}
\author{\fnms{Cl\'ement} \snm{Foucart}\thanksref{a},\ead[label=e1]{foucart@math.univ-paris13.fr}}
\author{\fnms{Chunhua} \snm{Ma}\thanksref{b}\ead[label=e2]{mach@nankai.educ.cn}}

\address[a]{Laboratoire Analyse, G\'eom\'etrie \& Applications, UMR 7539
Institut Galil\'ee Universit\'e Paris 13,
99 avenue J.B. Cl\'ement 93430 Villetaneuse
FRANCE\\
\printead{e1}}
\address[b]{
School of Mathematical Sciences and LPMC, Nankai University, Tianjin 300071 P. R. CHINA
\\
\printead{e2}}

\runauthor{C. Foucart, C. Ma}

\affiliation{Universit\'e Paris 13 and Nankai University}

\end{aug}

\begin{abstract}
The long-term behavior of flows of continuous-state branching processes are characterized through subordinators and extremal processes. The extremal processes arise in the case of supercritical processes with infinite mean and of subcritical processes with infinite variation. The jumps of these extremal processes are interpreted as specific initial individuals whose progenies overwhelm the population. These individuals, which correspond to the records of a certain Poisson point process embedded in the flow, are called super-individuals. They radically increase the growth rate to $+\infty$ in the supercritical case, and slow down the rate of extinction in the subcritical one.
\end{abstract}
\begin{abstract}[language=french]
Les comportements en temps long des flots de processus de branchement en temps et espace continus sont caract\'eris\'es par des subordinateurs et des processus extr\'emaux.  Les processus extr\'emaux apparaissent dans le cas des processus sur-critiques de moyenne infinie et des processus sous-critiques \`a variation infinie. Les sauts de ces processus extr\'emaux sont interpr\'et\'es comme des individus initiaux sp\'ecifiques dont les descendances envahissent la population. Ces individus, qui correspondent aux instants de records d'un certain processus ponctuel de Poisson, sont appel\'es super-individus. Ils augmentent de fa\c{c}on radicale la vitesse de divergence dans le cas sur-critique et diminuent celle d'extinction dans le cas sous-critique.
\end{abstract}
\end{frontmatter}

\section{Introduction and main result}\label{CBI}
We consider branching processes in continuous-time with continuous-state space (CSBPs) as defined by Ji{\v{r}}ina \cite{Jirina} and Lamperti \cite{Lamperti2},\cite{Lamperti1}. These processes are the continuous analogues of Galton-Watson Markov chains. A feature of the continuous-state space is that the population can become extinguished asymptotically while maintaining a positive size at any time. The process, in this case, is said to be \textit{persistent}. Grey in \cite{Grey} and Bingham \cite{BINGHAM1976217} have studied the long-term behavior of CSBPs. It is shown in \cite{Grey} that any CSBP with finite mean or finite variation can be linearly renormalized to converge almost-surely. Duquesne and Labb\'e in \cite{DuqLab} have generalized this result by showing that any \textit{flow} of CSBPs with finite mean or finite variation can be renormalized to converge towards a subordinator. This convergence corresponds to the natural idea that all individuals have progenies that grow or decline at the same scale. We show in this article that in the case of CSBPs with infinite mean or infinite variation, the flow can be renormalized (in a non-linear way) to converge towards the partial records of a Poisson point process. The limit process has therefore a very different nature than that in the case of finite mean or finite variation. An intuitive explanation is that the initial individuals have progenies that grow or decline at different rates. The limit process will take into account only the progeny of the individual whose growth rate is \textit{maximal}. An important example of persistent CSBP with infinite mean and infinite variation is the CSBP of Neveu, whose branching mechanism is $$\Psi(u)=u\log u=\int_{0}^{+\infty}(e^{-ux}-1+ux1_{\{x\leq 1\}})x^{-2}\ddr x.$$
Let $(X_t(x), t\geq 0)$ be a Neveu's CSBP with initial value $x$. A well-known result of Neveu \cite{Neveu}, states that for any fixed $x$ \begin{equation}\label{Neveu}
e^{-t}\log X_t(x)\underset{t\rightarrow +\infty}{\longrightarrow} Z(x) \text{ a.s.}
\end{equation}
where $Z(x)$ has a Gumbel law : for all $z\in \mathbb{R}$, $\mathbb{P}(Z(x)\leq z)=e^{-xe^{-z}}$. 
We show that for any CSBP starting from a fixed initial value $x$, which is non-explosive and has infinite mean or is persistent and has infinite variation, one can find a \textit{non-linear} renormalisation, in the same vein as (\ref{Neveu}), that converges towards a certain random variable $Z(x)$. This non-linear renormalisation reflects that the population grows or declines at a super-exponential rate. The problem of finding a renormalisation in the case of Galton-Watson processes with infinite mean has been considered by many authors, we refer to Grey \cite{Grey2}, Barbour and Schuh \cite{MR540789} and the references therein. We will adapt Grey's method to the continuous-state space setting. Following the seminal idea of Bertoin and Le Gall in \cite{LGB0}, a \textit{continuous population model} can be defined by considering a flow of subordinators $(X_t(x),t\geq 0 , x\geq 0)$. Formally, the progeny at time $t$ of the individual $x$ is given by $\Delta X_t(x)=X_t(x)-X_t(x-)$ the jump at $x$ of the subordinator $X_t$. The main purpose of this paper is to investigate the limit process $(Z(x), x\geq 0)$ and to interpret it in terms of the population model.  We will see that $(Z(x), x\geq 0)$ is an \textit{extremal process} whose law is explicit in terms of the branching mechanism. For instance, in the Neveu case, the process in (\ref{Neveu}) is an extremal-$\Lambda$ process with $\Lambda(x)=\exp\left(-e^{-x}\right)$.  In order to give some insights for the interpretation, we borrow some ideas of Bertoin, Fontbona and Martinez in \cite{MR2455180}, Labb\'e \cite{MR3224288} and Duquesne and Labb\'e \cite{DuqLab}.  In \cite{MR2455180}, the authors show that in a supercritical CSBP with infinite variation, some initial individuals have progenies that tend to $+\infty$. These individuals are called prolific and are responsible for the infinite growth of the process. 
\bdefinition\label{prolific} The individual $x$ is said to be prolific if $\underset{t\rightarrow+\infty}{\lim}\Delta X_t(x)=+\infty$. Denote by $\mathcal{P}$ the set of prolific individuals
$$\mathcal{P}:=\{x>0; \underset{t\rightarrow+\infty}{\lim}\Delta X_t(x)=+\infty\}.$$
\edefinition
We shall see that in a non-explosive CSBP with infinite mean (with or without infinite variation) and in a persistent CSBP with infinite variation, some individuals have a progeny that overwhelms the total progeny of all individuals \textit{below} them (see Definition \ref{defsuper}). These individuals will be called \textit{super-individuals}. \bdefinition \label{defsuper} The individual $x$ is said to be a super-individual if $\underset{t\rightarrow +\infty}\lim\frac{\Delta X_{t}(x)}{X_{t}(x-)}=+\infty$ a.s. Denote by $\mathcal{S}$ the set of super-individuals \beqlb\label{super-individual}\mathcal{S}:=\left\{x>0; \underset{t\rightarrow +\infty}\lim \frac{\Delta X_{t}(x)}{X_{t}(x-)}=+\infty\right\}.\eeqlb
\edefinition We stress that there is an order between the super-individuals: if $x_1$ and $x_2$ are in $\mathcal{S}$ and $x_1<x_2$ then the progeny of $x_2$ overwhelms that of $x_1$, since $0\leq \frac{\Delta X_t(x_1)}{\Delta X_t(x_2)}\leq \frac{X_t(x_2-)}{\Delta X_t(x_2)} \underset{t\rightarrow +\infty}{\longrightarrow} 0$. In the supercritical case, we say that an individual is \textit{super-prolific}, if it is a prolific super-individual. We will see that only certain prolific individuals are super-prolific. In the subcritical case, since all initial individuals have progenies that get extinct (in finite time or not), no prolific individual may exist. However, when the process is persistent with infinite variation, super-individuals do exist. They are individuals whose progeny decays at a much slower rate than all individuals below them. 
The super-individuals in the subcritical case and the super-prolific individuals in the supercritical case will correspond to the jumps times of the extremal process $(Z(x), x\geq 0)$.
In the finite variation case and finite mean case, $\mathcal{S}$ is degenerate (empty or reduced to a single point) and there are basically no super-individuals. A Poisson construction of the flow of subordinators $(X_t(x), t\geq 0)$ is given in \cite{DuqLab} for all branching mechanisms $\Psi$.  In the infinite mean or infinite variation case, this Poisson construction, recalled in Section \ref{genealogy}, allows us to determine a Poisson point process $\mathcal{M}$, as shown in Lemma \ref{thethreeinfinitevar} and Lemma \ref{thethreeinfinitevarsub}, whose partial records correspond to $(Z(x), x\geq 0)$. We state now our main result.
\begin{theorem}\label{main1}  Consider  $(X_t(x), t\geq 0, x\geq 0)$ a flow of CSBPs$(\Psi)$ as defined in (\ref{poissoninfinitevariation}) and (\ref{poissonfinitevariation}). 
\begin{itemize}
\item[i)] Assume $\Psi'(0+)=-\infty$ and $\int_{0}\frac{\ddr u}{\Psi(u)}=-\infty$. Call $\rho$ the largest root of $\Psi$, fix $\lambda_0\in (0,\rho)$ and define $G(y):=\exp\left(-\int^{\lambda_0}_{y}\frac{\ddr u}{\Psi(u)}\right)$ on $(0,\rho)$. Then, almost-surely for all $x\geq 0$
$$e^{-t}G\left(\frac{1}{X_{t}(x)}\wedge \rho\right) \underset{t\rightarrow +\infty}\longrightarrow Z(x)=\sup_{x_i\leq x}Z_i$$
where $\mathcal{M}:=\sum_{i\in I}\delta_{(x_i, Z_i)}$ is a Poisson point process over $\mathbb{R}_+\times (0,\infty)$ with intensity $\ddr x \otimes \mu(\ddr z)$ whose tail is $\bar{\mu}(z)=G^{-1}(z)$ (the inverse function of $G$). Moreover,
$$\mathcal{S}\cap \mathcal{P}=\{x>0; \Delta Z(x)>0\} \text{ a.s.}$$
\item[ii)] Assume $\Psi'(0+)\geq 0$, $\Psi(u)/u \underset{u\rightarrow +\infty}{\longrightarrow} +\infty$ and $\int^{+\infty}\frac{\ddr u}{\Psi(u)}=+\infty$. Fix $\lambda_0\in (0,+\infty)$ and define $G(y):=\exp\left(-\int_{\lambda_0}^{y}\frac{\ddr u}{\Psi(u)}\right)$ on $(0,+\infty)$. Then, almost-surely for all $x\geq 0$
$$e^{t}G\left(\frac{1}{X_{t}(x)}\right)\underset{t\rightarrow +\infty}\longrightarrow Z(x)=\sup_{x_i\leq x}Z_i$$
where $\mathcal{M}:=\sum_{i\in I}\delta_{(x_i, Z_i)}$ is a Poisson point process over $\mathbb{R}_+\times (0,\infty)$ with intensity $\ddr x \otimes \mu(\ddr z)$ whose tail is $\bar{\mu}(z)=G^{-1}(z)$. Moreover, 
$$\mathcal{S}=\{x>0; \Delta Z(x)>0\} \text{ a.s.}$$
\end{itemize} 
\end{theorem}
The Poisson point process $\mathcal{M}$ represents the initial individuals with their asymptotic growth rates. In the supercritical case, we will see that non-prolific individuals have growth rates $Z_i$ equal to $0$. In the subcritical case, we shall see that there are infinitely many super-individuals near zero. We highlight that the function $G^{-1}$ can be written explicitly in terms of the cumulant of the CSBP$(\Psi)$.

In \cite{DuqLab}, the authors establish a classification of branching mechanims for which the population concentrates on the progeny of a single individual. This phenomenon is called \textit{Eve property}. More precisely, the population started from a \textit{fixed} size $x$ has the Eve property if there exists a random variable $\textsf{e}\in [0,x]$, such that
\beqlb\label{the eve}\frac{\Delta X_t(\textsf{e})}{X_{t}(x)}\underset{t\rightarrow \zeta_x}{\longrightarrow} 1 \text{ a.s}\eeqlb
where $\zeta_x:=\inf\{t\geq 0; X_t(x)\in \{0,+\infty\}\} \in \mathbb{R}_+\cup \{+\infty\}$. The individual $\textsf{e}$ is called the Eve. We refer to Definition 1.1 in \cite{MR3224288} and Definition 0.1 in \cite{DuqLab} for equivalent definitions. Unsurprisingly, the Eve property holds precisely when extremal processes arise (i.e.\;in the infinite variation or infinite mean case). From (\ref{super-individual}) and (\ref{the eve}), we see that the Eve $\textsf{e}$ must be a super-individual  whose progeny overwhelms those of individuals in $(\textsf{e},x]$. In our setting, the Eve of the population started with size $x$ is therefore characterized as the \textit{last} super-individual in $[0,x]$. In other words, the Eve is the individual whose growth is the fastest in the  non-explosive supercritical case with infinite mean or whose decay is the slowest in the persistent subcritical one with infinite variation. This corresponds to the record of the Poisson point process $\mathcal{M}$ in $[0,x]$. In particular, this \textit{forward-in-time} argument  allows us to follow the Eve of the population started from $x$ when $x$ evolves in the half-line. Labb\'e, in \cite{MR3224288}, has shown that when the population has an Eve, one can define a recursive sequence of Eves on which the population concentrates. We stress that  super-individuals and successive Eves are not the same individuals. Indeed, the successive Eves are i.i.d uniform in $[0,x]$ (see Proposition 4.13 in \cite{MR3224288}), whereas the super-individuals are ordered  and therefore not independent.

We wish to mention that Bertoin et al. in \cite{MR2455180} have shown that the number of prolific individuals, when time evolves, is an immortal branching process. This discrete process is related to the backbone decomposition which has been deeply studied by Berestycki et al \cite{Berestycki20111315}, Kyprianou et al \cite{MR3330813},\cite{MR2863035} and by Duquesne and Winkel \cite{Duquesne2007} in the framework of random trees. We will not address here the study of the number of super-individuals in time. Moreover, no spatial motion is taken into account in this work. We refer the reader to Fleischmann and Sturm \cite{MR2086012} and Fleischmann and Wachtel \cite{MR2238611} where superprocesses with Neveu's branching mechanism are studied. Lastly, the CSBP of Neveu and its limit (\ref{Neveu}) have been used in the study of Derrida's random energy model by Neveu \cite{Neveu}, Bovier and Kourkova \cite{BovierKurkova} and Huillet \cite{Huillet}. 

The paper is organized as follows. In Section \ref{preliminaries}, we recall the definition of a continuous-state branching process and some of its important properties. We describe the Poisson construction of the continuous population model (as in \cite{DuqLab}). Then we gather some results about extremal processes. In Section \ref{supercritical}, we start by a brief recall of the convergence of supercritical CSBPs with finite mean towards subordinators, established in \cite{DuqLab}. We deduce that in this case except the first prolific individual, there is no super-individual. Then, we focus on CSBPs with infinite mean and we establish the finite-dimensional convergence towards an extremal process (Theorem \ref{main11}). Afterwards we prove Theorem \ref{main1}-i) through three Lemmas: Lemma \ref{thethreeinfinitevar}, Lemma \ref{Zsup} and Lemma \ref{superprolific}. In Section \ref{subcritical}, we study the subcritical processes. The organisation is similar. Theorem \ref{main2} yields the convergence towards an extremal process characterized by its finite-dimensional marginal laws. Lemma \ref{lemmasub}, Lemma \ref{thethreeinfinitevarsub} and Lemma \ref{superresistant} are obtained similarly as in the supercritical case. By combining them with Theorem \ref{main2}, we obtain Theorem \ref{main1}-ii). In Subsection \ref{subsectionNeveu}, we treat the special case of the Neveu's CSBP. 
\section{Preliminaries}\label{preliminaries}
\textbf{Notation}. If $x$ and $y$ are two real numbers. We denote their maximum by $x\vee y$, and their minimum by $x\wedge y$. If $X$ and $Y$ are two random variables, $X\overset{d}{=}Y$ means that $X$ and $Y$ have the same law.  
\subsection{Continuous-state branching processes} 
Our main references are Chapter 12 of Kyprianou's book \cite{Kyprianoubook} and Chapter 3 of Li's book \cite{Li}. A positive Markov process $(X_t(x),t\geq 0)$ with $X_0(x)=x\geq 0$ is a continuous-state branching process in continuous time (CSBP for short) if for any $y\in \mathbb{R}_{+}$
\begin{equation}\label{branching}
(X_t(x+y), t\geq 0)\stackrel{d}{=}(X_t(x), t\geq 0)+(\tilde X_t(y), t\geq 0)
\end{equation}where $(\tilde X_t(y), t\geq 0)$ is an independent copy of $(X_t(y), t\geq 0)$. 
The branching property (\ref{branching}) ensures the existence of a map $\lambda\mapsto v_{t}(\lambda)$, called cumulant, such that for all $\lambda\geq 0$ and all $t,s\geq 0$
\begin{equation}\label{cumulant}
\mathbb{E}[e^{-\lambda X_{t}(x)}]=\exp(-xv_{t}(\lambda))  \text{ and } v_{s+t}(\lambda)=v_{s}\circ v_t(\lambda).
\end{equation}
Moreover, there exists a unique function $\Psi$ of the form
\begin{equation} \label{Levykhintchine} \Psi(q)=\frac{\sigma^{2}}{2}q^{2}+\gamma q+\int_{0}^{+\infty}\left(e^{-qx}-1+qx1_{\{x\leq 1\}}\right)\pi(\ddr x)
\end{equation}
with $\gamma\in\mathbb{R}$, $\sigma \geq 0$, and $\pi$ a $\sigma$-finite measure carried on $\mathbb{R}_{+}$ satisfying
\begin{equation*}
\int_{0}^{+\infty} (1\wedge x^{2}) \pi(\ddr x)<+\infty
\end{equation*}
such that the map $t\mapsto v_{t}(\lambda)$ is the unique solution to the integral equation
\begin{equation}\label{cumulantintegral} \forall t\in [0,+\infty), \forall \lambda \in (0,+\infty)/\{\rho\}, \quad \int_{v_{t}(\lambda)}^{\lambda}\frac{\ddr z}{\Psi(z)}=t
\end{equation}
where $\rho=\inf\{z>0, \Psi(z)\geq 0\}\in [0,+\infty]$. The process is said to be supercritical if $\Psi'(0+)\in [-\infty,0)$ (in which case $\rho\in (0,+\infty]$), subcritical if $\Psi'(0+)\in(0, +\infty)$ and critical if $\Psi'(0+)=0$ (in these last two cases $\rho=0$). 
\begin{theorem}[Grey, \cite{Grey}]\label{Grey} 
Consider $(X_{t}(x), t\geq 0)$ a CSBP$(\Psi)$ started from $x$.
\begin{itemize}
\item[i)] For any $x\geq 0$, $$\mathbb{P}(\underset{t\rightarrow +\infty}\lim X_{t}(x)=0)=1-\mathbb{P}(\underset{t\rightarrow +\infty}\lim X_{t}(x)=+\infty)=e^{-x\rho}$$
One has $\rho=+\infty$ if and only if $-\Psi$ is the Laplace exponent of a subordinator. In this case, the process is non-decreasing and tends to $+\infty$ almost-surely. 
\item[ii)] The process is almost-surely not absorbed at $0$ if and only if 
\begin{equation}
\label{NonAbsorption}
\int^{+\infty}\frac{\ddr u}{|\Psi(u)|}=+\infty \quad \text{(persistence).}
\end{equation}
If $\int^{+\infty}\frac{\ddr u}{\Psi(u)}<+\infty$, the limits $\bar{v}_{t}:=\underset{\lambda \rightarrow +\infty}{ \lim }v_{t}(\lambda)\in (0,+\infty)$  for any $t\geq 0$ and $\bar{v}:=\underset{t \rightarrow +\infty}{\lim \downarrow} \bar{v}_{t}$ exist. Moreover for all $t\geq 0$  $\frac{\ddr }{\ddr t}\bar{v}_{t}=-\Psi(\bar{v}_{t})$ with $\bar{v}_{0}=+\infty$ and  $$\mathbb{P}(X_t(x)=0)=\exp{(-x\bar{v}_t)}\text{ and }
\mathbb{P}(\exists t\geq 0: X_t(x)=0)=\exp{(-x\bar{v})}.$$
\item[iii)] The process is almost-surely not absorbed in $+\infty$ if and only if
\begin{equation}
\label{non-explosionHypothesis}
\int_{0}\frac{\ddr u}{|\Psi(u)|}=+\infty \quad \text{(non-explosion).}
\end{equation}
If $\int_{0}\frac{\ddr u}{|\Psi(u)|}<+\infty$, the limits  $\underline{v}_{t}:=\underset{\lambda \rightarrow 0}{ \lim }v_{t}(\lambda)\in (0,+\infty)$ for any $t\geq 0$ and $\underline{v}:=\underset{t \rightarrow +\infty}{\lim \uparrow} \underline{v}_{t}$ exist. Moreover for all $t\geq 0$, $\frac{\ddr }{\ddr t}\underline{v}_{t}=-\Psi(\underline{v}_{t})$ with $\underline{v}_{0}=0$ and
 $$\mathbb{P}(X_t(x)=+\infty)=1-\exp{(-x\underline{v}_t)}\text{ and } \mathbb{P}(\exists t\geq 0: X_t(x)=+\infty)=1-\exp{(-x\underline{v})}.$$
\end{itemize}
\end{theorem}
We classify now the mechanisms $\Psi$ according to their behaviour near $0$ and $+\infty$.  For any $\Psi$ as in (\ref{Levykhintchine}), $$\Psi'(0+)=
\underset{u\rightarrow 0}\lim\frac{\Psi(u)}{u}=\gamma-\int_{1}^{+\infty}x\pi(\ddr x)\in[-\infty,+\infty) \quad \text{(mean)}.
$$
One can show from (\ref{cumulant}) that $\mathbb{E}(X_{t}(x))=xe^{-\Psi'(0+)t}$  for all $t\geq 0$, this leads to the following classification.
\begin{itemize}
\item[-] If $\Psi'(0+)\in (-\infty,0)$, the process has a finite mean and $\int_{0}\frac{\ddr u}{\Psi(u)}=-\infty$. Therefore the process does not explode almost-surely and goes to $+\infty$ with probability $1-e^{-x\rho}$.
\item[-] If $\Psi'(0+)=-\infty$ and $\int_{0}\frac{\ddr u}{\Psi(u)}=-\infty$ then the process has an infinite mean, does not explode almost-surely and goes to $+\infty$ with probability $1-e^{-x\rho}$. The Neveu's CSBP provides an example of non-explosive CSBP with infinite mean.  
\item[-]If $\int_{0}\frac{\ddr u}{\Psi(u)}\in (-\infty,0)$, the process explodes continuously to $+\infty$ with probability $1-e^{-x\rho}$ if $\underline{v}<+\infty$, with probability $1$ if $\underline{v}=+\infty$. 
\end{itemize}
For any $\Psi$ as in (\ref{Levykhintchine}), $$\textbf{d}:=\underset{u\rightarrow +\infty}\lim\frac{\Psi(u)}{u}=+\infty 1_{\{\sigma>0\}}+\gamma+\int_{0}^{1}x\pi(\ddr x)\in (-\infty,+\infty] \quad \text{(variation).}$$
\begin{itemize}
\item[-] If $\textbf{d}\in \mathbb{R}$, then the process has finite variation sample paths (we will say that $\Psi$ is of finite variation) and $\int^{\infty}\frac{\ddr u}{\Psi(u)}=+\infty$. Therefore the process is persistent (not absorbed at $0$ almost-surely) and goes to $0$ with probability $e^{-x\rho}$ ($\rho=+\infty$ if $\textbf{d}\leq 0$).
\item[-] If $\textbf{d}=+\infty$ and $\int^{+\infty}\frac{\ddr u}{\Psi(u)}=+\infty$, then the process has infinite variation sample paths, is persistent and goes to $0$ with probability $e^{-x\rho}$.
\item[-] If $\textbf{d}=+\infty$ and $\int^{+\infty}\frac{\ddr u}{\Psi(u)}<+\infty$, then the process has infinite variation sample paths and is absorbed at $0$ with probability $e^{-x\rho}$.
\end{itemize}
Note that $\textbf{d}\in \mathbb{R}$ if and only if $\sigma=0$ and $\int_0^1 u\pi(\ddr u)<+\infty$. In this case (\ref{Levykhintchine}) can be rewritten as
\begin{equation}\label{finitevariationpsi} \Psi(u)=\textbf{d}u-\int_{0}^{+\infty}\pi(\ddr r)(1-e^{-ur}),
\end{equation} 
Unless explicitly specified, we shall always work under the non-explosion condition (\ref{non-explosionHypothesis}).
\begin{lemma} \label{v}
Denote by $\lambda\mapsto v_{-t}(\lambda)$ the inverse of $\lambda\mapsto v_t(\lambda)$. This is a strictly increasing function, well-defined on $[0,\bar{v}_{t})$. For all $s,t\in \mathbb{R}^+$, if $0\leq \lambda<\bar{v}_{s+t}$, then $$v_{-(s+t)}(\lambda)=v_{-s}\circ v_{-t}(\lambda).$$ Moreover by (\ref{cumulantintegral}), for all $t\geq 0$ and $\lambda<\bar{v}_{t}$
\begin{equation}\label{v-}
\int_{v_{-t}(\lambda)}^{\lambda}\frac{\ddr u}{\Psi(u)}=\int_{v_{-t}(\lambda)}^{v_{t}(v_{-t}(\lambda))}\frac{\ddr u}{\Psi(u)}=-t.
\end{equation}
In the supercritical case, $\Psi'(0+)<0$, $v_{-t}(\lambda)\underset{t\rightarrow +\infty}{\longrightarrow} 0$ and for any $\lambda',\lambda\in (0,\rho)$, we have 
\begin{equation}\label{vsup}\frac{v_{-t}(\lambda)}{v_{-t}(\lambda')}\underset{t\rightarrow +\infty}{\longrightarrow} \exp\left(\Psi'(0+)\int_{\lambda'}^{\lambda}\frac{\ddr u}{\Psi(u)}\right).
\end{equation}
In the subcritical case, $\Psi'(0+)\geq 0$, $v_{-t}(\lambda)\underset{t\rightarrow +\infty}{\longrightarrow} +\infty$. One has $\textbf{d}\in (0,+\infty]$ and for any $\lambda,\lambda'>0$, we have \begin{equation}\label{vsub}
\frac{v_{-t}(\lambda)}{v_{-t}(\lambda')}\underset{t\rightarrow +\infty}{\longrightarrow} \exp\left(\textbf{d}\int_{\lambda'}^{\lambda}\frac{\ddr u}{\Psi(u)}\right).
\end{equation}
\end{lemma}
\begin{proof} We refer the reader to \cite{Grey}. We show (\ref{vsup}) and (\ref{vsub}). By (\ref{cumulant}), (\ref{cumulantintegral}) and (\ref{v-}), one can show that for all $t\in \mathbb{R}$, $\frac{\ddr}{\ddr \lambda}v_{t}(\lambda)=\frac{\Psi(v_{t}(\lambda))}{\Psi(\lambda)}$, and  therefore for any $\lambda, \lambda'$
$$\frac{v_{-t}(\lambda)}{v_{-t}(\lambda')}=\exp\left(\int_{\lambda'}^{\lambda}\frac{\ddr }{\ddr u}\log v_{-t}(u)\ddr u\right)=\exp\left(\int_{\lambda'}^{\lambda}\frac{\Psi(v_{-t}(u))}{v_{-t}(u)}\frac{\ddr u}{\Psi(u)}\right).$$ 
In the supercritical case, $v_{-t}(u)\underset{t\rightarrow +\infty}{\longrightarrow} 0$ and by convexity of $\Psi$, $\frac{\Psi(v_{-t}(u))}{v_{-t}(u)}$ decreases towards $\Psi'(0+)\in [-\infty,0)$ as $t$ goes to $+\infty$ . In the subcritical case, $v_{-t}(u)\underset{t\rightarrow +\infty}{\longrightarrow} +\infty$ and by convexity of $\Psi$, $\frac{\Psi(v_{-t}(u))}{v_{-t}(u)}$ increases towards $\textbf{d}\in (-\infty,+\infty]$ as $t$ goes to $+\infty$. The limits (\ref{vsup}) and (\ref{vsub}) are obtained by monotone convergence. 
\end{proof}
The following lemma holds for all non-explosive supercritical CSBPs and persistent subcritical ones. Note that in the case of a subcritical non-persistent CSBP, $\bar{v}=0$ and the following lemma is degenerate. We refer to Theorem 3.2.1 in \cite{Li} for a proof.
\begin{lemma}[Grey's martingale \cite{Grey}] \label{martingale} For all $x\geq 0$, and $\lambda\in (0,\bar v)$, the process \[(M^{\lambda}_{t}(x), t\geq 0):=(\exp\left(-v_{-t}(\lambda)X_{t}(x)\right), t\geq 0)\] is a positive martingale.
\end{lemma}
\subsection{Continuous population model}\label{genealogy}
As noticed in \cite{LGB0}, the branching property (\ref{branching}) allows one to apply the Kolmogorov extension theorem and to define on some probability space a flow of subordinators $(X_t(x), x\geq 0, t\geq 0)$ such that 
\begin{enumerate}
\item[i)] for all $t\geq 0$ $(X_t(x), x\geq 0)$ is a subordinator with Laplace exponent $\lambda \mapsto v_{t}(\lambda)$ 
\item[ii)] for any $y\geq x$, $(X_t(y)-X_t(x), t\geq 0)$ is a CSBP$(\Psi)$ started from $y-x$, independent of $(X_t(x), t\geq 0)$.
\end{enumerate}
This provides a genuine continuous population model: the individual $a$ living at time $0$ has for descendant $b$ at time $t$, if $$X_{t}(a-)<b<X_{t}(a).$$
Duquesne and Labb\'e in \cite{DuqLab} (see Theorem 0.2 and Theorem 1.8 in \cite{DuqLab}), provide a construction for any mechanism $\Psi$ of the flow $(X_t(x), x\geq 0, t\geq 0)$ via a Poisson point process on the space of c\`adl\`ag trajectories (see also Dawson and Li in \cite{DawsonLi} for an approach with Poisson driven stochastic differential equations). 

In the case of infinite variation, the Laplace exponent $\lambda \mapsto v_{t}(\lambda)$ is driftless and thus takes the form $v_{t}(\lambda)=\int_{(0,+\infty]}\ell_{t}(\ddr x)(1-e^{-\lambda x})$ for a certain L\'evy measure $\ell_{t}$ on $\mathbb{R}^{+}\cup \{+\infty\}$ such that $\int_{(0, +\infty]}(1\wedge x)\ell_{t}(\ddr x)<+\infty$.  As noticed in Chapter 3 of \cite{Li}, $(\ell_t, t>0)$ is an entrance law for the semi-group of the CSBP$(\Psi)$. This yields the existence of a measure $N_{\Psi}$ (called cluster measure in \cite{DuqLab}, and canonical measure in \cite{Li}) on the space $\mathcal{D}$ of c\`adl\`ag paths from $\mathbb{R}^{\star}_+$ to $\mathbb{R}_{+}\cup \{+\infty\}$ such that for any non-negative function $F$
\begin{equation}\label{clustermeasure}
N_\Psi(F(X_{t+\cdot}); X_t>0)=\int_{(0,+\infty]}\ell_{t}(\ddr x)\mathbb{E}_{x}^{\Psi}(F) \text{ and } N_\Psi(X_0>0)=0
\end{equation} 
and for any non-negative function $F,G$, the Markov property holds:
\begin{equation}\label{markov}
N_\Psi[F(X_{\cdot\wedge t})G(X_{t+\cdot}); X_t>0]=N_{\Psi}[F(X_{\cdot\wedge t})\mathbb{E}_{X_t}[G]; X_t>0].
\end{equation}
Consider a Poisson point process $\mathcal{N}=\sum_{i\in I}\delta_{(x_i, X^i)}$ over $\mathbb{R}_+\times \mathcal{D}$ with intensity $\ddr x \otimes N_{\Psi}(\ddr X)$ and set for all $x\geq 0$ and $t>0$, 
\begin{equation}\label{poissoninfinitevariation} X_t(x)=\sum_{x_i\leq x}X^i_t, \end{equation}
with $X_0(x)=x$. The flow $(X_t(x), x\geq 0, t\geq 0)$ defined by (\ref{poissoninfinitevariation}) satisfies the properties i) and ii). 

In the case of finite variation, one can construct a flow $(X_t(x), x\geq 0, t\geq 0)$ in a Poisson manner as in \cite{DuqLab} (see Equation 1.25). Recall $\textbf{d}:=\lim_{u\rightarrow +\infty}\frac{\Psi(u)}{u}$.  Consider a Poisson point process $\mathcal{N}=\sum_{i\in I}\delta_{(x_i, t_i, X^i)}$ over $\mathbb{R}_+\times \mathbb{R}_+\times \mathcal{D}$ with intensity $\ddr x \otimes e^{-\textbf{d}t}\ddr t\otimes \int_{0}^{+\infty}\pi(\ddr r)\mathbb{P}_{r}^{\Psi}(\ddr X)$. Set \begin{equation}\label{poissonfinitevariation} X_{t}(x)=e^{-\textbf{d}t}x+\sum_{x_{i}\leq x}1_{\{t_i\leq t\}}X^{i}_{t-t_i}.
\end{equation}
The flow $(X_t(x), x\geq 0, t\geq 0)$ defined by  (\ref{poissonfinitevariation}) satisfies the properties i) and ii). 

In both finite or infinite variation case, for any $x\in \mathbb{R}_+$, we define $\Delta X_t(x):= X_{t}(x)-X_{t}(x-)$, this represents the progeny at time $t$ of the individual $x$ at time $0$. By the Poisson construction, any individual with a non zero progeny at a certain time belongs to $\{x_i, i\in I\}$.
\begin{remark} One can understand $N_{\Psi}$ as the L\'evy measure of the path-valued subordinator $(X_{t}(x), x\geq 0)_{t\geq 0}$. Roughly speaking, we add CSBPs with the same mechanism $\Psi$ started at individuals with zero mass. We refer to Li \cite{MR3161481} for a recent work on path-valued processes.  
\end{remark}
In the rest of the paper, if not otherwise specified, the flows that we consider are all constructed from Poisson point processes as in (\ref{poissoninfinitevariation}) and (\ref{poissonfinitevariation}).
\subsection{Extremal processes}\label{extremal}
We gather here some of the fundamental properties of extremal processes that can be found in Chapter 4, Section 3 in Resnick \cite{Res87}. Let $F$ be a probability distribution function with a given support $(s_l,s_o)\subset \overline{\mathbb{R}}$. A real-valued process $(M_x,x\geq 0)$ is an extremal-$F$ process 
if for any $0\leq x_1\leq ...\leq x_n$ and $(z_1,...,z_n) \in \mathbb{R}^n$,
\begin{equation}\label{extremalprocess}
\mathbb{P}(M_{x_1}\leq z_1,M_{x_2}\leq z_2,...,M_{x_n}\leq z_n)=F^{x_1}(z'_1)F^{x_2-x_1}(z'_2)...F^{x_n-x_{n-1}}(z'_n)
\end{equation}
where $z'_i=\wedge_{k=i}^{n} z_{k}$ for all $i\geq 1$. Any extremal-$F$ process $(M_x, x\geq 0)$ has the following properties:
\begin{itemize}
\item[i)] $(M_x, x\geq 0)$ is stochastically continuous.
\item[ii)] $(M_x, x\geq 0)$ has a c\`adl\`ag version.
\item[iii)] $(M_x, x\geq 0)$ has a version with non-decreasing paths such that $\lim_{x\rightarrow +\infty}M_x=s_o$ and $\lim_{x\rightarrow 0}M_x=s_l$ almost-surely.
\item[iv)]  $(M_x, x\geq 0)$ is a Markov process with for $x>0, y>0$, \beqlb\label{sg}
\mathbb{P}(M_{x+y}\leq z \ |M_{x}=v)= \begin{cases} F^{y}(z) & \text{ if } z\geq v\\
0 &\text{ if }z<v.
\end{cases}
\eeqlb
For all $x\geq 0$, set $Q(x)=-\log F(x)$. The parameter of the exponential holding time in state $x$ is $Q(x)$, and the process jumps from $x$ to $(x,y]$ with probability $1-\frac{Q(y)}{Q(x)}$. The only possible instantaneous state is $s_l$ and it is instantaneous if and only if $F(s_l)=0$.
\end{itemize} 
Any process $(M_x,x\geq 0)$ verifying 
\begin{equation}\label{maxid}
\begin{cases}
&\mathbb{P}(M_x\leq z)=F(z)^{x} \text{ for all } z\in \mathbb{R} \text{ and } x\geq 0\\
&M_{x+y}=M_{x}\vee M'_{y} \text{ a.s.} \text{ for all }x,y\in \mathbb{R}_+ \end{cases}
\end{equation}
where $M'_y$ is independent of $(M_u, 0\leq u\leq x)$  and $M'_y\overset{d}=M_y$, satisfies (\ref{extremalprocess}) and is therefore an extremal-$F$ process. A constructive approach of extremal processes is given by the records of a Poisson point process. Let $\mu$ be a $\sigma$-finite measure on $(0,+\infty)$, and consider a Poisson point process $\mathcal{P}=\sum_{i\in I}\delta_{(x_i,Z_i)}$ with intensity $dx\otimes \mu$. The process $(M_x, x\geq 0)$ defined by $M_x=\sup_{x_{i}\leq x}Z_{i}$ is a c\`adl\`ag extremal-$F$ process with for all $z\in \mathbb{R}$, $F(z)=\exp\left(-\bar{\mu}(z)\right)$ where $\bar{\mu}(z)=\mu(z,+\infty)$. We highlight that the state $0$ is instantaneous if the intensity measure $\mu$ is infinite. The positive extremal processes will play an important role in the sequel. They correspond to the records of Poisson point processes over $\mathbb{R}_+\times \mathbb{R}_+$.  

An interesting feature of extremal processes lies in their \textit{max-infinite divisibility}. Namely, for any integer $m$, 
\begin{equation}\label{maxid2}(M_x, x\geq 0)\overset{d}{=} (\underset{i\in [|1,m|]}{\max M^i_x}, x\geq 0)
\end{equation}
where $(M^i_x,x\geq 0)_{i\in [|1,m|]}$ are i.i.d extremal-$F^{1/m}$ processes. Indeed, 
\begin{align*}
\mathbb{P}(\underset{i\in [|1,m|]} \max M^i_{x_1}\leq z_1,... ,\underset{i\in [|1,m|]} \max M^i_{x_n}\leq z_n)&=\left(F^{\frac{x_1}{m}}(z'_1)F^{\frac{x_2-x_1}{m}}(z'_2)...F^{\frac{x_n-x_{n-1}}{m}}(z'_n)\right)^m\\
&=\mathbb{P}(M_{x_1}\leq z_1,M_{x_2}\leq z_2,...,M_{x_n}\leq z_n).
\end{align*}
We refer the reader to Dwass \cite{dwass1966} and Resnick and Rubinovitch \cite{APR:10159662} for more details about extremal processes.
 
\section{Supercritical processes}\label{supercritical} 

\subsection{Prolific individuals}
Consider $\Psi$ a supercritical mechanism, namely such that $\Psi'(0+)\in [-\infty,0)$. Consider a flow of CSBPs $(X_{t}(x), x\geq 0, t\geq 0)$ with mechanism $\Psi$ and recall the notion of prolific individual defined in Definition \ref{prolific}. The following lemma generalizes Lemma 2 in \cite{MR2455180} to the finite variation case.
\begin{lemma}\label{Bertoin} The set of prolific individuals
$$\mathcal{P}:=\{x\in \mathbb{R}_+; \lim_{t\rightarrow +\infty}\Delta X_{t}(x)=+\infty\}$$
is almost-surely non-empty. If $\rho<+\infty$, $(\#\mathcal{P}\cap [0,x], x\geq 0)$ is a Poisson process with intensity $\rho$. If $\rho=+\infty$ then $\mathcal{P}=\{x_{i}, i\in I\}$.  
\end{lemma}
\begin{proof} For any $a_n>0$, the set $\mathcal{P}\cap [0,a_n]$ is non-empty if and only if $(X_{t}(a_n), t\geq 0)$ is not extinguishing. Consider a non-decreasing sequence $(a_n)_{n\geq 1}$ such that $a_n\underset{n\rightarrow +\infty}{\longrightarrow}+\infty$:  $\mathbb{P}(\mathcal{P}\cap [0,a_n]\neq \emptyset)=1-e^{-\rho a_n}\underset{n\rightarrow +\infty}{\longrightarrow} 1$. Therefore, $\mathbb{P}(\mathcal{P}\neq \emptyset)=1$. Assume $\rho<+\infty$, in the infinite variation case, one has $N_\Psi(X_t\underset{t\rightarrow +\infty}\longrightarrow +\infty, X_s>0)=\int_{(0,+\infty]}\ell_s(\ddr x)\mathbb{P}_{x}^{\Psi}(X_t\underset{t\rightarrow +\infty}{\longrightarrow} +\infty)=\int_{(0,+\infty]}\ell_s(\ddr x)(1-e^{-x\rho})=v_{s}(\rho)=\rho$. By letting $s$ to $0$, we see that the restriction of $\mathcal{P}$ to the atoms $x_i$ such that $X^i_t\underset{t\rightarrow +\infty}{\longrightarrow} +\infty$ is a Poisson point process with intensity $\rho\ddr x$. In the finite variation case, if $\rho<+\infty$, then $\textbf{d}>0$ and $$\int_{0}^{+\infty}e^{-\textbf{d}t}\ddr t\int_{0}^{+\infty}\pi(\ddr r)\mathbb{P}_{r}^{\Psi}(X_{u-t}\underset{u\rightarrow +\infty}{\longrightarrow} +\infty)=\frac{1}{\textbf{d}}\int_{0}^{+\infty}\pi(\ddr r)(1-e^{-r\rho})=\frac{-\Psi(\rho)+\textbf{d}\rho}{\textbf{d}}=\rho.$$
The restriction of $\mathcal{P}$ to the atoms $x_i$ such that $X^i_{t-t_i}\underset{t\rightarrow +\infty}{\longrightarrow} +\infty$ is therefore also a Poisson point process over $\mathbb{R}_+$ with intensity $\rho\ddr x$.
Assume $\rho=+\infty$, then by Theorem \ref{Grey}-i), $\mathbb{P}_r(X_t \underset{t\rightarrow +\infty}{\not \longrightarrow} +\infty)=0$, and thus $\int_{0}^{+\infty}\pi(\ddr r)\mathbb{P}^{\Psi}_r(X_t\underset{t\rightarrow +\infty}{ \not \longrightarrow} +\infty)=0$. Therefore $\mathcal{P}=\{x_i; i\in I\}$.\end{proof}
\subsection{CSBPs with finite mean and subordinators.} 
We briefly study the case of a branching mechanism with finite mean to show that the notion of super-prolific individual is degenerate. The following proposition is essentially a rewriting of Proposition 2.1 and Lemma 2.2 in \cite{DuqLab}, but is important before treating the infinite mean case. 
\begin{Prop}[Proposition 2.1 and Lemma 2.2 in \cite{DuqLab}]\label{finite} Suppose $\Psi'(0+)\in(-\infty,0)$ and fix $\lambda\in (0,\rho)$. Consider a flow of CSBPs($\Psi$) $(X_t(x),x\geq 0,t\geq 0)$ defined as in (\ref{poissoninfinitevariation}) or (\ref{poissonfinitevariation}). There exists a c\`adl\`ag driftless subordinator $(W^{\lambda}(x), x\geq 0)$ with Laplace exponent $\theta\mapsto v_{\frac{\log(\theta)}{-\Psi'(0+)}}(\lambda)$ such that almost-surely for any $x>0$, $$v_{-t}(\lambda)X_{t}(x) \underset{t\rightarrow +\infty}\longrightarrow W^{\lambda}(x) \text{ and } v_{-t}(\lambda)X_{t}(x-) \underset{t\rightarrow +\infty}\longrightarrow W^{\lambda}(x-).$$ The L\'evy measure of $(W^{\lambda}(x), x\geq 0)$ has total mass $\rho\in (0,+\infty]$. Moreover 
$$\mathcal{P}=\{x>0; W^{\lambda}(x)>W^{\lambda}(x-)\}.$$
If $\rho=+\infty$, $\mathcal{S}\cap \mathcal{P}=\emptyset$ a.s.; if $\rho<+\infty$,  $\mathcal{S}\cap \mathcal{P}=\{x^{\star}\}$ a.s.\,with $x^{\star}$ the time of the first jump of $(W^{\lambda}(x),x\geq 0)$. 
\end{Prop}
\begin{remark} The infinite divisibility of $(W^{\lambda}(x), x\geq 0)$ ensures that the random variable $W^{\lambda}(x)$ has the same law as the sum of $n$ copies of $W^{\lambda}(x/n)$:
\[W^{\lambda}(x)\overset{d}{=}W^1(x/n)+...+W^n(x/n).\]
In a loose sense, the progeny of the individuals $[0,x]$ grows as the sum of prolific individuals progenies in $[0,x]$. 
\end{remark}
\begin{proof} We only show that $\mathcal{P}=\{x>0; W^{\lambda}(x)>W^{\lambda}(x-)\}$ and $\mathcal{S}\cap \mathcal{P}=\{x^{\star}\}  \text{ a.s}$ if $\rho<\infty$, $\mathcal{S}\cap \mathcal{P}=\emptyset$ if $\rho=\infty$.  Lemma 2.2 in \cite{DuqLab} states that $(W^{\lambda}(x), x\geq 0)$ is a subordinator such that  
$\underset{t\rightarrow \infty}{\lim} v_{-t}(\lambda)\Delta X_t(x)=\Delta W^{\lambda}(x)$. By Lemma \ref{v}, $v_{-t}(\lambda)\underset{t\rightarrow +\infty}\longrightarrow 0$. This implies $$\{x>0, \Delta W^{\lambda}(x)>0\}\subset\{x>0, \lim_{t\rightarrow \infty}\Delta X_{t}(x)=\infty\}.$$ Moreover, when $\rho<\infty$, $(W^{\lambda}(x), x\geq 0)$ is a compound Poisson process with intensity $\rho$. By Lemma \ref{Bertoin}, for any $n\in \mathbb{N}$, $\#\{0<x<n; \Delta W^{\lambda}(x)>0\}\leq \#\{0<x<n; \lim_{t\rightarrow \infty}\Delta X_{t}(x)=\infty\}$. For any $n$, these two random variables have the same Poisson law with parameter $\rho n$ and therefore for all $n$, almost-surely
$$\{0<x<n, \Delta W^{\lambda}(x)>0\}=\{0<x<n, \lim_{t\rightarrow \infty}\Delta X_{t}(x)=\infty\}.$$
By letting $n$ to infinity, we get that $\mathcal{P}=\{x>0; \Delta W^{\lambda}(x)>0\}$ almost-surely. If $\rho=\infty$, one has $\mathbb{P}_r(v_{-t}(\lambda)X_t\underset{t\rightarrow \infty}{\longrightarrow} 0)=0$ and thus $\{x>0, \Delta W^{\lambda}(x)>0\}=\{x_i, i\in I\}=\mathcal{P}.$ Moreover, if $\rho<\infty$, then $\mathcal{S}\cap \mathcal{P}=\left\{x>0; \Delta W^{\lambda}(x)>0 \text{ and } \frac{\Delta W^{\lambda}(x)}{W^{\lambda}(x-)}=+\infty\right\}=\{x^{\star}\}$, since $W^{\lambda}_{x^\star-}=0$. If $\rho=\infty$, then for all $x>0$, $\underset{t\rightarrow \infty}{\lim} \frac{\Delta X_{t}(x)}{X_{t}(x-)}=\frac{\Delta W^{\lambda}(x)}{W^{\lambda}(x-)}<\infty$,
therefore $\mathcal{S}\cap \mathcal{P}=\emptyset.$ 
\end{proof}
\subsection{CSBPs with infinite mean and extremal processes}
\begin{theorem}\label{main11} Suppose $\Psi'(0+)=-\infty$ and $\int_{0}\frac{\ddr u}{\Psi(u)}=-\infty$. Fix $\lambda_0\in(0,\rho)$, and define $G(y):=\exp\left(-\int^{\lambda_0}_{y}\frac{\ddr u}{\Psi(u)}\right)$ for $y\in (0,\rho)$.Then, for all $x\geq 0$, almost-surely
$$e^{-t}G\left(\frac{1}{X_{t}(x)}\wedge \rho\right) \underset{t\rightarrow +\infty}\longrightarrow \tilde{Z}(x),$$ where $(\tilde{Z}(x), x\geq 0)$ is a positive extremal-$F$ process (in the sense of (\ref{extremalprocess})) with $F(z)=\exp\left(-G^{-1}(z)\right)$ for $z \in [0,+\infty]$, and $G^{-1}(z)=v_{\log \left(\frac{1}{z}\right)}(\lambda_0)$ for all $z \in [0,+\infty]$.
\end{theorem}
\begin{example}\label{example1}
Consider the Neveu's mechanism $\Psi(u)=u\log u$ for which $\rho=1$. For all $t\in \mathbb{R}$, $v_{t}(\lambda)=\lambda^{e^{-t}}=e^{-\log \left(\frac{1}{\lambda}\right) e^{-t}}$.
Fix $\lambda_0= \frac{1}{e}$, one has $G(z)=\log(1/z)$ and the process $(\tilde{Z}(x), x\geq 0)$ is a positive extremal-$F$ process with for $z\geq 0$, $F(z)=e^{-e^{-z}}$ and for $z<0$, $F(z)=0$. As $F(0)>0$, the state $0$ is not instantaneous. 
\end{example}
\begin{remark} By applying the branching property at time $t\geq 0$ in Theorem \ref{main11}, one can see that for all $x\geq 0$, $\tilde{Z}(x)\overset{d}{=}e^{-t}\tilde{Z}'(X_t(x))$, where $(\tilde{Z}'(x), x\geq 0)$ has the same distribution as $(\tilde{Z}(x), x\geq 0)$ and is independent of $(X_s(x): 0\leq s\leq t, x\geq0)$. This result was observed by Cohn and Pakes in \cite{MR0488337} for Galton-Watson processes with infinite mean. The max-infinite divisibility of the process $(\tilde{Z}(x),x\geq 0)$ ensures that the random variable $\tilde{Z}(x)$ has the same law as the maximum of $n$ independent copies of $\tilde{Z}(x/n)$: \[\tilde{Z}(x)\overset{d}{=}\max(\tilde{Z}^{1}(x/n),...,\tilde{Z}^ {n}(x/n)).\]
In a loose sense, the infinite mean of the process transforms the sum into a maximum.
\end{remark}
We start the proof of Theorem \ref{main11} by two lemmas, the first one shows notably the slow variation of $G$ at $0$.
\begin{lemma}\label{slowvariation} 
If $\Psi'(0+)=-\infty$, and $\int_{0}\frac{\ddr u}{\Psi(u)}=-\infty$, the map $G:y\mapsto \exp\left(-\int_{y}^{\lambda_0}\frac{\ddr u}{\Psi(u)}\right)$ is continuous, non-increasing, goes from $[0,\rho]$ to $[0, +\infty]$ and is slowly varying at $0$. Moreover, for all $y\in (0,\rho)$, $G'(y)=\frac{G(y)}{\Psi(y)}$ with $G(\lambda_0)=1$ and
$G^{-1}(z)=v_{\log \left(\frac{1}{z}\right)}(\lambda_0)$ for $z \in [0,+\infty]$.
\end{lemma}
\begin{proof} 
Since $\Psi$ is nonpositive on $(0, \rho)$, and $\int_{\rho}^{\lambda_0}\frac{\ddr u}{\Psi(u)}=+\infty$ then $G(y)\underset{y\rightarrow \rho}{\longrightarrow} 0$.  Moreover since $\int_{0}\frac{\ddr u}{\Psi(u)}=-\infty$, then $G(y)\underset{y\rightarrow 0}{\longrightarrow} +\infty$. By definition $\int_{G^{-1}(z)}^{\lambda_0}\frac{\ddr u}{\Psi(u)}=\log\left(\frac{1}{z}\right)$ and by (\ref{cumulantintegral}) $G^{-1}(z)=v_{\log\left(\frac{1}{z}\right)}(\lambda_0)$.
By assumption $\Psi'(0+)=-\infty$, then for any fixed $\theta>0$, there exists  $|b|$ arbitrarily large, such that for a small enough $y$ 
$$\left \lvert\int_y^{\theta y}\frac{\ddr u}{|\Psi(u)|}\right \lvert \leq \left\lvert \int_{y}^{\theta y}\frac{\ddr u}{|b|u}\right\lvert=\frac{|\log(\theta)|}{|b|}$$ this tends to $0$ as $b\rightarrow +\infty $. Therefore $\underset{y\rightarrow 0}{\lim}\int_{y}^{\theta y}\frac{\ddr u}{\Psi(u)}=0$ and $G$ is slowly varying at $0$. \end{proof}
\begin{lemma}\label{WW}
Let $x>0$ and $\lambda\in (0,\bar{v})$. The limit $W^{\lambda}(x):=\underset{t\rightarrow +\infty}{\lim} v_{-t}(\lambda)X_{t}(x)$ exists almost-surely in $\mathbb{R}_+\cup\{+\infty\}$. Moreover, for all $x\geq 0$ $$\P(W^{\lambda}(x)=0)=1-\P(W^{\lambda}(x)=+\infty)=\exp\left(-x\lambda\right).$$
\end{lemma}
\begin{proof} Lemma \ref{martingale} and the martingale convergence theorem applied to $(M_t^{\lambda})_{t\geq 0}$ ensure that $v_{-t}(\lambda)X_{t}(x)$ converges almost-surely as $t$ goes to infinity towards a random variable $W^{\lambda}(x)$ with values in $[0,+\infty]$. Let $0\leq \lambda<\bar{v}$ and $\theta\geq 0$. One has for all $x\geq 0$
\begin{equation}\label{laplaceW} \mathbb{E}[e^{-\theta v_{-t}(\lambda)X_{t}(x)}]=\exp\left(-xv_{t}(\theta v_{-t}(\lambda))\right).
\end{equation}
By Lemma \ref{v}, $v_{-t}(\lambda)\underset{t\rightarrow +\infty}\longrightarrow 0$. For all $\theta> 0$, and $t$ such that $v_{-t}(\lambda), \theta v_{-t}(\lambda)\in (0,\rho)$;
we have by (\ref{cumulantintegral}) and Lemma \ref{v} 
\begin{align}\label{equav}
\int_{\lambda}^{v_{t}(\theta v_{-t}(\lambda))}\frac{\ddr z}{\Psi(z)}&=\int_{\lambda}^{v_{-t}(\lambda)}\frac{\ddr z}{\Psi(z)}+\int_{v_{-t}(\lambda)}^{\theta v_{-t}(\lambda)}\frac{\ddr z}{\Psi(z)}+\int_{\theta v_{-t}(\lambda)}^{v_t(\theta v_{-t}(\lambda))}\frac{\ddr z}{\Psi(z)} \nonumber\\
&=\int_{v_{-t}(\lambda)}^{\theta v_{-t}(\lambda)}\frac{\ddr z}{\Psi(z)}.
\end{align}
Fix any positive constant $\theta\neq1$. For any $b>0$, there exists a large enough $t$ such that $|\Psi(z)|\geq |b|z$ for  all $z\in((\theta\wedge1) v_t(\lambda), (\theta\vee1)v_t(\lambda))$. Therefore by Lemma \ref{slowvariation} and (\ref{equav}),
\begin{equation*}\lim_{t\rightarrow +\infty} \left| \int_{\lambda}^{v_t(\theta v_{-t}(\lambda))}\frac{\ddr z}{\Psi(z)} \right| \leq \lim_{t\rightarrow +\infty} \int_{v_{-t}(\lambda)}^{\theta v_{-t}(\lambda)}\frac{\ddr z}{|\Psi(z)|}=0.\end{equation*}
Then $\lim_{t\rightarrow +\infty} \int_{\lambda}^{v_t(\theta v_{-t}(\lambda))}\frac{\ddr z}{\Psi(z)}=0$ and thus for all $\theta>0$
\begin{equation}\label{theta}
\lim_{t\to+\infty}v_{t}(\theta v_{-t}(\lambda))=\lambda.
\end{equation}
The limit in (\ref{laplaceW}) as $t$ tends to $+\infty$ equals $e^{-x\lambda}$ and does not depend on $\theta$. The random variable $W^{\lambda}(x)$ is thus equal to $0$ or $+\infty$ with probability $1$. \end{proof}
\begin{proof}[Proof of Theorem \ref{main11}] The arguments for the almost-sure convergence are adapted from those on pages 711-712 in \cite{Grey2}. Fix $x>0$. By Lemma \ref{martingale} and Lemma \ref{WW}, there exists $\Omega_0\in\mathcal{F}$ such that $\mathbb{P}(\Omega_0)=1$ and on $\Omega_0$, for any $q\in(0,\rho)\cap\mathbb{Q}$, $v_{-t}(q)X_{t}(x)\underset{t\rightarrow +\infty}\longrightarrow W^q(x)$ with $W^q(x)$ taking values in $\{0,+\infty\}$. Since $q\mapsto v_{-t}(q)$ is increasing then by definition, $W^{q}_x\leq W^{q'}_x$ if $q\leq q'$.  Then $\{W^q(x): q\in(0,\rho)\cap\mathbb{Q}\}$ steps up from $0$ to $+\infty$ at some
random threshold $\Lambda_x$ and is otherwise constant. Define $\Lambda_x:=\inf\{q\in (0,\rho)\cap\mathbb{Q}: W^{q}(x)=+\infty\}\in\mathbb{R}_+$.
On $\Omega_0$, for any $\lambda\in[0,\rho)$,
\beqnn
\{\Lambda_x\leq\lambda\}=\lim_{q\in\mathbb{Q}, q\downarrow\lambda}\{W^q(x)=+\infty\} \quad \mbox{and}\quad \{\Lambda_x=\rho\}
=\lim_{q\in\mathbb{Q}, q\uparrow\rho}\{W^q(x)=0\}.
\eeqnn
Then $\Lambda_x$ is a random variable. By Lemma \ref{WW}, we have $\mathbb{P}(\Lambda_x\leq\lambda)=1-e^{-x\lambda}$ for $\lambda\in[0,\rho)$ and $\mathbb{P}(\Lambda_x=\rho)=e^{-x\rho}$, which implies that 
$\{\Lambda_x=\rho\}=\{X_{t}(x)\underset{t\rightarrow +\infty}{\longrightarrow} 0\}$ a.s.  We then work deterministically on $\Omega_0$ to show that for any $\lambda\in (0,\rho)$, 
\beqlb\label{0-infty}
v_{-t}(\lambda)X_{t}(x)\underset{t\rightarrow +\infty}\longrightarrow  W^{\lambda}(x)=\begin{cases} 0 & \text{ if } \Lambda_x>\lambda\\
+\infty &\text{ if }\Lambda_x<\lambda.
\end{cases}
\eeqlb
Let $\lambda\in(0,\rho)$. If $\Lambda_x>\lambda$, there exists some $q'\in\mathbb{Q}$ such that $\Lambda_x>q'>\lambda$. Since $v_{-t}(q')X_t(x)\underset{t\rightarrow +\infty}\longrightarrow 0$ and $v_{-t}(q')X_t(x)\geq v_{-t}(\lambda)X_t(x)$ therefore $v_{-t}(\lambda)X_t(x)\underset{t\rightarrow +\infty}\longrightarrow 0$. If $\Lambda_x<\lambda$, there exists some $q''\in\mathbb{Q}$ such that  $\Lambda_x<q''<\lambda$. Since $v_{-t}(q'')X_t(x)\leq v_{-t}(\lambda)X_t(x)$ and $v_{-t}(q'')X_t(x)\underset{t\rightarrow +\infty}\longrightarrow +\infty$, therefore $v_{-t}(\lambda)X_t(x)\underset{t\rightarrow +\infty}\longrightarrow +\infty$. 

Assume $\Lambda_x<\rho$. Choose $\lambda'$ and $\lambda''$ such that $\lambda',\lambda''\in (0,\rho)\cap\mathbb{Q}$ and $\lambda'<\Lambda_x<\lambda''$, by definition if $t$ is large enough,  \begin{center}
$v_{-t}(\lambda'')X_{t}(x)\geq 1$ and $v_{-t}(\lambda')X_{t}(x)\leq 1,$
\end{center} thus
$$G(v_{-t}(\lambda'))\geq G(1/X_{t}(x))\geq G(v_{-t}(\lambda'')).$$
Recall (\ref{cumulantintegral}) and (\ref{v-}), this yields $G(v_{-t}(\lambda'))=e^{-\int_{\lambda'}^{\lambda_0}\frac{\ddr x}{\Psi(x)}}e^{t}=G(\lambda')e^{t}$, and then
$$G(\lambda')\geq e^{-t}G(1/X_{t}(x))\geq G(\lambda'').$$
Since $\lambda'$ and $\lambda''$ are arbitrarily close to $\Lambda_x$, and $G$ is continuous, we get $$e^{-t}G\left(\frac{1}{X_{t}(x)}\right) \underset{t\rightarrow +\infty}\longrightarrow G(\Lambda_x)\quad \mathbb{P}\text{-almost surely on } \{\Lambda_x<\rho\}.$$
If $\Lambda_x=\rho$, $1/X_{t}(x)\underset{t\rightarrow +\infty}{\longrightarrow} +\infty$, and $$G\left(\frac{1}{X_{t}(x)}\wedge \rho\right)\underset{t\rightarrow +\infty}\longrightarrow G(\rho)=0\quad \mathbb{P}\text{-almost surely on } \{\Lambda_x=\rho\}.$$ Define 
\begin{equation}\label{defZ}
\tilde{Z}(x)=G(\Lambda_x).
\end{equation}
The one-dimensional law of $\tilde{Z}(x)$ follows readily. In order to avoid cumbersome notations, we only show that $(\tilde{Z}(x), x\geq 0)$ satisfies (\ref{extremalprocess}) for the two-dimensional marginals. Let $z_1, z_2 \in \mathbb{R}_+$. By (\ref{0-infty}) the events $\{\tilde{Z}(x_1)< z_1, \tilde{Z}(x_2)< z_2\}$ and $\{W^{G^{-1}(z_{1})}(x_{1})=0, W^{G^{-1}(z_{2})}(x_{2})=0\}$ are identical. Let $\lambda_1=G^{-1}(z_1)$ and $\lambda_2=G^{-1}(z_2)$, then
\begin{align*}
&\mathbb{P}(W^{\lambda_1}(x_1)=0, W^{\lambda_2}(x_{2})=0)\\&= \mathbb{E}[e^{-W^{\lambda_1}(x_{1})-W^{\lambda_2}(x_2)}]\\
&=\underset{t\rightarrow +\infty}{\lim}\mathbb{E}\left[\exp\left(-v_{-t}(\lambda_1)X_{t}(x_1)-
v_{-t}(\lambda_2)X_{t}(x_{2})\right)\right]\\
&=\underset{t\rightarrow +\infty}{\lim}\mathbb{E}\left[\exp\left(-( v_{-t}(\lambda_1)+
v_{-t}(\lambda_2))X_{t}(x_1)-v_{-t}(\lambda_2)(X_{t}(x_{2})-X_{t}(x_{1}))\right)\right]\\
&=\underset{t\rightarrow +\infty}{\lim}\mathbb{E}[\exp(-(v_{-t}(\lambda_1)+ v_{-t}(\lambda_2))X_{t}(x_{1}))]\mathbb{E}[\exp(-v_{-t}(\lambda_2)(X_{t}(x_{2})-X_{t}(x_{1})))]\\
&=\underset{t\rightarrow +\infty}{\lim}\exp\left(-x_{1}v_{t}(v_{-t}(\lambda_1)+v_{-t}(\lambda_2))\right)\exp\left(-(x_2-x_1)v_{t}(v_{-t}(\lambda_2))\right).
\end{align*}
By definition $v_{t}(v_{-t}(\lambda_2))=\lambda_2$. With no loss of generality assume $\lambda_1< \lambda_2$, by Lemma \ref{v}, since $\Psi'(0+)=-\infty$, then $\frac{v_{-t}(\lambda_1)}{v_{-t}(\lambda_2)}\underset{t\rightarrow +\infty}{\longrightarrow} 0$. Fix any $\theta>1$. For $t$ large enough
$$v_{t}(v_{-t}(\lambda_2))\leq v_{t}(v_{-t}(\lambda_1)+v_{-t}(\lambda_2))\leq v_{-t}(\theta v_{-t}(\lambda_2)).$$
Recall (\ref{theta}), since $\lim_{t\to+\infty}v_{t}(\theta v_{-t}(\lambda))=\lambda$, therefore
$v_{t}(v_{-t}(\lambda_1)+v_{-t}(\lambda_2))\underset{t\rightarrow +\infty}{\longrightarrow} \lambda_1 \vee \lambda_2.$ Thus
\begin{align*}
\mathbb{P}(\tilde{Z}(x_1)< z_1, \tilde{Z}(x_2)< z_2)&=e^{-x_{1}G^{-1}(z_1)\vee G^{-1}(z_2)}e^{-(x_2-x_1)G^{-1}(z_2)}\\
&=e^{-x_{1}G^{-1}(z_1\wedge z_2)}e^{-(x_2-x_1)G^{-1}(z_2)}.
\end{align*}
\end{proof}
\begin{remark} An alternative route to see that $(\tilde{Z}(x), x\geq 0)$ is an extremal-$F$ process is to verify (\ref{maxid}) instead of (\ref{extremalprocess}). By applying Theorem \ref{main11} to the CSBP $(X_t(x+y)-X_t(x), t\geq 0)$, we get that $\underset{t\rightarrow +\infty}\lim e^{-t}G\left(\frac{1}{X_t(x+y)-X_t(x)}\right)=:\tilde{Z}(x,x+y)$ exists almost-surely, has the same law as $\tilde{Z}_y$ and is independent of $\tilde{Z}(x)$. It is readily checked that $\tilde{Z}(x+y)$ and $\tilde{Z}(x) \vee \tilde{Z}(x,x+y)$ have the same law. Since $G$ is non-increasing, one has $\tilde{Z}(x+y)\geq \tilde{Z}(x) \vee \tilde{Z}(x,x+y)$ a.s. Therefore $\tilde{Z}(x+y)=\tilde{Z}(x) \vee \tilde{Z}(x,x+y)$ a.s. 
\end{remark}
\begin{Prop}\label{superexponential1}
If $\Psi'(0+)=-\infty$ and there are $\lambda>0$ and $\alpha>0$ such that 
$\left\lvert\int_{0}^{\lambda}\left(\frac{1}{\Psi(u)}-\frac{1}{\alpha u \log u}\right)\ddr u \right\lvert<+\infty,$
then $G(1/y)\underset{y\rightarrow +\infty}\sim k_{\lambda}\log(y)^{1/\alpha},$ with $k_{\lambda}$ a positive constant. Fix $x>0$, on the event $\{X_t(x)\underset{t\rightarrow +\infty}{\longrightarrow}+\infty\}$, \[\log X_t(x)\underset{t\rightarrow +\infty}{\sim} e^{\alpha t}k_{\lambda}^{-\alpha}\tilde{Z}(x)^{\alpha} \text{ a.s}.\]
\end{Prop}
\begin{proof} By definition of $G$, \begin{align*}
\frac{G(1/y)}{(\log y)^{1/\alpha}}&=\exp\left(-\int_{\frac{1}{y}}^{\lambda}\frac{\ddr u}{\Psi(u)}-\frac{1}{\alpha}\log \log y\right)\\
&=\exp\left(-\int_{\frac{1}{y}}^{\lambda}\left(\frac{1}{\Psi(u)}-\frac{1}{\alpha u \log u}\right)\ddr u-\frac{1}{\alpha}\log \log \frac{1}{\lambda}\right)\underset{y\rightarrow +\infty}{\longrightarrow} k_{\lambda}\in (0,+\infty).
\end{align*}
By Theorem \ref{main11}, $e^{-t}(\log X_t(x))^{1/\alpha}\underset{t\rightarrow +\infty}{\longrightarrow} k_{\lambda}^{-1}\tilde{Z}(x)$ a.s.
\end{proof}

The following proposition shows how to associate an extremal process to a flow of explosive CSBPs through the explosion times. 
\begin{Prop}[Theorem 0.3-i) in \cite{DuqLab} for $\zeta=\zeta_\infty$]\label{explosiveflow} Consider a flow of CSBPs$(\Psi)$, $(X_t(x), x\geq 0, t\geq 0)$, with $\Psi$ such that $\int_{0}\frac{\ddr u}{|\Psi(u)|}<+\infty$. Define $\xi_0=+\infty$ and $\xi_x:=\inf\{t>0; X_t(x)=+\infty\}.$ The process $(Z(x), x\geq 0):=(1/\xi_{x}, x\geq 0)$ is an extremal-$F$ process with $F(z)=\exp(-\underline{v}_{\frac{1}{z}})$. For all $i\in I$, set $Z_i:=1/\xi_i$ with $\xi_i:=\inf\{t\geq 0; X^{i}_t=+\infty\}$ in the infinite variation case and $\xi_i:=t_i+\inf\{t\geq t_i; X^{i}_{t-t_i}=+\infty\}$ in the finite variation case. The point process $\mathcal{M}:=\sum_{i\in I}\delta_{(x_i,Z_i)}$ is a Poisson point process with intensity $\ddr x\otimes \mu(\ddr z)$ with $\bar{\mu}(z)=\underline{v}_{\frac{1}{z}}$ and almost-surely for all $x\geq 0$, $Z(x)=\sup_{x_{i}\leq x} Z_i$. Moreover 
$$\mathcal{S}:=\left\{x>0, \exists t>0; \Delta X_t(x)=+\infty \text{ and } X_{t}(x-)<+\infty\right\}=\{x>0; \Delta Z(x)>0\}.$$
\end{Prop}
\begin{proof}  By Theorem \ref{Grey}-iii), $\mathbb{P}(\xi_x>1/z)=e^{-x\underline{v}_{\frac{1}{z}}}$ where $t\mapsto\underline{v}_t$ is the unique solution to 
$\frac{\ddr \underline{v}_{t}}{\ddr t}=-\Psi(\underline{v}_{t})$ and $\underline{v}_0=0$. Plainly, $\xi_{x+y}=\xi_{x}\wedge \xi_{x,x+y}$ with $\xi_{x,x+y}:=\inf\{t\geq 0; X_t(x+y)-X_t(x)=+\infty\}$. The random variable $\xi_{x,x+y}$ is independent of $(\xi_{u}, 0\leq u\leq x)$ and has the same law as $\xi_{y}$. Therefore, the process $(Z(x), x\geq 0)$ satisfies (\ref{maxid}) and is an extremal-$F$ process with $F(z)=e^{-\underline{v}_{\frac{1}{z}}}$. Assume $\Psi$ of finite variation, the intensity of $\mathcal{M}$ is $\ddr x\otimes \mu(\ddr z)$ where 
\begin{align*}
\bar{\mu}(z)&=\int_{0}^{+\infty}e^{-\textbf{d}t}\ddr t\int_{0}^{+\infty}\pi(\ddr r)\mathbb{P}^{\Psi}_r\left(\frac{1}{t+\xi}>z\right)\\
&=\int_{0}^{1/z}e^{-\textbf{d}t}\ddr t\int_{0}^{+\infty}\pi(\ddr r)\mathbb{P}^{\Psi}_r\left(\frac{1}{t+\xi}>z\right) \text{ since } \mathbb{P}^{\Psi}_r\left(\frac{1}{t+\xi}>z\right)=0 \text{ if } t>1/z\\
&=\int_{0}^{1/z}e^{-\textbf{d}t}\ddr t\int_{0}^{+\infty}\pi(\ddr r)\left(1-e^{-r\underline{v}_{1/z-t}}\right) \text{ by Theorem \ref{Grey}-iii)}\\
&=\int_{0}^{1/z}e^{-\textbf{d}t}\ddr t\left(-\Psi(\underline{v}_{1/z-t})+\textbf{d}\underline{v}_{1/z-t}\right) \text{ since } \Psi \text{ has the form (\ref{finitevariationpsi}})\\
&=\int_{0}^{1/z}e^{-\textbf{d}t}\ddr t\left(-\frac{\ddr \underline{v}_{1/z-t}}{\ddr t}+\textbf{d}\underline{v}_{1/z-t}\right)=\left[-e^{-\textbf{d}t}\underline{v}_{1/z-t}\right]^{t=1/z}_{t=0}=\underline{v}_{1/z}.
\end{align*}
By definition, $\xi_x= \underset{x_i\leq x}{\inf}\{t\geq 0; X^i_{(t-t_i)_+}=+\infty\}=\underset{x_i\leq x}{\inf} \xi_i$ and then $Z(x)=\underset{x_i\leq x}{\sup} Z_i$. In the infinite variation case, by Equation \ref{clustermeasure}, 
\begin{align*}
N_{\Psi}(Z>z; X_s>0)&=\int_{(0,+\infty]}\ell_{s}(\ddr x)\left(1-\mathbb{P}^{\Psi}_{x}(Z<z)\right)=\int_{(0,+\infty]}\ell_{s}(\ddr x)\left(1-\mathbb{P}^{\Psi}_{x}(\xi>1/z)\right) \\
&=\int_{(0,+\infty]}\ell_{s}(\ddr x)\left(1-e^{-x\underline{v}_{1/z}}\right)=v_{s}(\underline{v}_{1/z})\underset{s\rightarrow 0}{\longrightarrow}\underline{v}_{1/z}. 
\end{align*}
By definition, $\xi_x= \underset{x_i\leq x}{\inf}\{t\geq 0; X^i_t=+\infty\}=\underset{x_i\leq x}{\inf} \xi_i$ and then $Z(x)=\underset{x_i\leq x}{\sup} Z_i$.
\end{proof}
\begin{example} Let $\alpha\in (0,1)$ and $\Psi(u)=-c_{\alpha} u^{\alpha}$ with $c_\alpha=\frac{1}{1-\alpha}$. Consider $(X_t(x), t\geq 0, x\geq 0)$ a flow of CSBPs$(\Psi)$. By Theorem \ref{Grey}-iii), for all $x$, the process $(X_t(x), t\geq 0)$ is explosive. The explosion time $\xi_x$ of the process $(X_t(x), t\geq 0)$ has a Weibull law with parameter $\frac{1}{1-\alpha}$. By Proposition \ref{explosiveflow}, the process $(Z(x), x\geq 0)$ is an extremal-$F$ process with $F$ the probability distribution function of a Fr\'echet law with parameter $\frac{1}{1-\alpha}\in (1,+\infty)$, that is to say $F(z)=e^{-z^{-\frac{1}{1-\alpha}}}$ for all $z\geq0$. 
\end{example}
\subsection{Proof of Theorem \ref{main1}-i)}\label{proofmaini}
The arguments provided in the sequel could be simplified in the case $\rho<+\infty$ merely because only finitely many individuals in $[0,x]$ are prolific. We shall not distinguish the cases $\rho<+\infty$ and $\rho=+\infty$ and the arguments will hold also in the subcritical case with infinite variation.
\begin{lemma}\label{lemma} Suppose that  $(X_t, t\geq 0)$ and $(Y_t, t\geq 0)$ are two independent CSBPs$(\Psi)$ on the same probability space, satisfying the conditions of Theorem \ref{main11}, with initial value $X_0$ and $Y_0$ respectively.  Then
$$e^{-t}G\left(\frac{1}{X_t}\wedge \rho\right)\underset{t\rightarrow +\infty}\longrightarrow \Gamma_1, \quad e^{-t}G\left(\frac{1}{Y_t}\wedge \rho\right)\underset{t\rightarrow +\infty}\longrightarrow \Gamma_2 \text{ and } e^{-t}G\left(\frac{1}{X_t+Y_t}\wedge \rho\right)\underset{t\rightarrow +\infty}\longrightarrow \Gamma_1 \vee\Gamma_2 \text{ a.s.}$$
where $\Gamma_1$ and $\Gamma_2$ are independent such that $\Gamma_1=\Gamma_2$ if and only if $\Gamma_1=\Gamma_2=0$ and
$$\Gamma_1\overset{d}{=}\bar{Z}(X_0),\quad  \Gamma_2\overset{d}{=}\bar{Z}(Y_0) \text{ and }\Gamma_1 \vee \Gamma_2 \overset{d}{=}\bar{Z}(X_0+Y_0),$$ where $(\bar{Z}(x), x\geq0)$ is an extremal-$F$ process, independent of $X_0$ and $Y_0$. 
\end{lemma}
\begin{proof} 
Let $(\bar{Z}(x), x\geq0)$ be an extremal-$F$ process, independent of $X_0$ and $Y_0$. Conditionally given $X_0$ and $Y_0$, the processes $(X_t, t\geq 0)$ and $(Y_t, t\geq 0)$ are independent CSBPs with same mechanism $\Psi$ started respectively from $X_0$ and $Y_0$. The branching property ensures that $(X_t+Y_t, t\geq 0)$ is a CSBP$(\Psi)$ started from $X_0+Y_0$. By applying Theorem \ref{main11}, there exists three random variables $\Gamma_1$, $\Gamma_2$ and $\Gamma_3$ such that almost-surely 
$$e^{-t}G\left(\frac{1}{X_t}\wedge \rho\right)\underset{t\rightarrow +\infty}\longrightarrow \Gamma_1, \quad e^{-t}G\left(\frac{1}{Y_t}\wedge \rho\right)\underset{t\rightarrow +\infty}\longrightarrow \Gamma_2 \text{ and } e^{-t}G\left(\frac{1}{X_t+Y_t}\wedge \rho\right)\underset{t\rightarrow +\infty}\longrightarrow \Gamma_3 \text{ a.s.}$$
with the same law respectively as $\bar{Z}(X_0)$, $\bar{Z}(Y_0)$ and $\bar{Z}(X_0+Y_0)$. By (\ref{extremalprocess}), for any $x$ and $y$, $\bar{Z}(x+y)\overset{d}{=}\bar{Z}(x)\vee \bar{Z}'(y)$ with $\bar{Z}'(y)$ independent of $\bar{Z}(x)$, therefore $\Gamma_3\overset{d}{=}\Gamma_1\vee \Gamma_2$. Moreover for any $t\geq 0$, $$e^{-t}G\left(\frac{1}{X_t+Y_t}\wedge \rho\right)\geq e^{-t}G\left(\frac{1}{X_t}\wedge \rho\right) \vee e^{-t}G\left(\frac{1}{Y_t}\wedge \rho\right).$$
Thus $\Gamma_3 \geq \Gamma_1 \vee \Gamma_2 \text{ a.s.}$ which, with the equality in law, entails  $\Gamma_3 =\Gamma_1 \vee \Gamma_2 \text{ a.s.}$ Since, conditionally on $X_0$ and $Y_0$, the laws of $\Gamma_1$ and $\Gamma_2$ have no atoms in $(0,+\infty)$, one has $\Gamma_1=\Gamma_2$ a.s if and only if $\Gamma_1=\Gamma_2=0$ a.s. 
\end{proof}
\begin{lemma}\label{thethreeinfinitevar} Assume $\Psi'(0)=-\infty$, $\int_{0}\frac{\ddr u}{\Psi(u)}=-\infty$ and $\Psi$ is of infinite variation. Consider a flow of CSBPs $(X_t(x), t\geq 0, x\geq 0)$ as in (\ref{poissoninfinitevariation}). Then almost-surely for all $i\in I$, 
\begin{equation}\label{threelimits}
Z_i:=\underset{t\rightarrow +\infty}{\lim}e^{-t}G\left(\frac{1}{X^i_t}\wedge \rho\right).
\end{equation}
The point process 
$\mathcal{M}:=\sum_{i\in I}\delta_{(x_i, Z_i)}$  is a Poisson point process with intensity $\ddr x\otimes \mu(\ddr z)$ where $\bar{\mu}(z)=G^{-1}(z)$. \end{lemma}
\begin{proof}
Consider the flow of CSBP  $(X_t(x), t\geq 0, x\geq0)$ defined by (\ref{poissoninfinitevariation}). First we will prove that the following event
$$\Omega_0:=\left\{\text{ for all } i\in I,\lim_{t\rightarrow+\infty}e^{-t}G\Big(\frac{1}{X^i_t}\wedge\rho\Big)
 \text{ exist  in } [0,+\infty)\right\}$$
has probability $1$. Let $(s_l, l\geq 1), (\epsilon_k, k\geq 1)$ two sequences of positive real numbers decreasing towards $0$. For any fixed $l$ and $k$, define $I_{l,k}=\{i \in I, X^{i}_{s_{l}}\geq \epsilon_k\}.$
Set  
$$\Omega_{0}^{l,k}:=\left\{\text{ for all } i\in I_{l,k},\lim_{t\rightarrow+\infty}e^{-t}G\Big(\frac{1}{X^i_t}\wedge\rho\Big) \text{ exist  in } [0,+\infty)\right\}.$$
For any fixed $l$, set $\Omega_{0}^{l}:=\bigcap_{k\geq 1}\Omega_{0}^{l,k}$. One has $\Omega_{0}^{l+1}\subset \Omega_{0}^{l}$ and  
$\Omega_{0}=\bigcap_{l,k\geq 1}\Omega_{0}^{l,k}.$ 
Observe that 
\beqnn
\{(x_i,(X^{i}_{s_{l}+t})_{t\geq 0}), i\in I_{l,k}\}=\{(U^{l,k}_n,(V_{t}^{(n),l,k})_{t\geq 0}), n=1,2,\ldots\},
\eeqnn
where $(U^{l,k}_n, n\geq 1)$ are the arrival times of the Poisson process 
\beqlb \label{poisson process 1}
(N_{l,k}(x), x\geq 0):=(\#\{x_i\leq x, X^i_{s_l}>\epsilon_k\}, x\geq 0)
\eeqlb
 whose parameter is $\ell_{s_l}((\epsilon_k,+\infty))$, and $(V^{(n),l,k}_{\cdot})_{n\geq 1}$ is a sequence of i.i.d.\;CSBPs$(\Psi)$ with initial 
value $V^{(n),l,k}_0$ whose law is ${\ell_{s_l}(\ddr x;x>\epsilon_k)}/{\ell_{s_l}((\epsilon_k,+\infty))}$. 
It follows from Lemma \ref{lemma} that for all $i\in I_{l,k}$, almost-surely,
\beqlb\label{two limits}
e^{-t}G\left(\frac{1}{X^i_t}\wedge \rho\right)=e^{-s_l}e^{-(t-s_l)}G\left(\frac{1}{X^i_{s_l+(t-s_l)}}\wedge \rho\right)\underset{t\rightarrow +\infty}{\longrightarrow} Z_i,
\eeqlb
where $\{(x_i, Z_i), i\in I_{l,k}\}=\{(U^{l,k}_n, Z^{l,k}_n), n=1,2,\ldots\}$ and $(Z^{l,k}_n)_{n\geq 1}$ is a sequence of i.i.d.\;non-negative random variables independent of the sequence $(U^{l,k}_n,n\geq 1)$.
For each fixed $n$, $Z^{l,k}_n$ has the same distribution as $e^{-s_l}\bar{Z}({V_0^{(n),l,k}})$.
Here $(\bar{Z}(x), x\geq0)$ is an extremal process, independent of $\sigma\{
X^i_u(x), 0\leq u\leq s_l,\ x\geq0,\ i\in I\}$ and thus independent of the sequences $(V_0^{(n),l,k})_{n \geq 1}$.
By (\ref{two limits}), we have that $\mathbb{P}(\Omega_0^{l,k})=1$. Therefore one has $\mathbb{P}(\Omega_0)=
\displaystyle\lim_{k\rightarrow +\infty \atop l\rightarrow +\infty}\mathbb{P}(\Omega_0^{l,k})=1$. Then on $\Omega_0$ define the point process $\mathcal{M}:=\sum_{i\in I}\delta_{(x_i, Z_i)}$. It is easy to see that for $\lambda>0$ and $z>0$,
\beqnn
\mathbb{E}\big[e^{-\lambda\mathcal{M}((0,u]\times(z,+\infty))}\big]=
\lim_{k\rightarrow +\infty \atop l\rightarrow +\infty}\mathbb{E}\big[e^{-\lambda\sum_{i\in I_{l,k}}1_{\{x_i\leq u, Z_i>z\}}}\big]
=\lim_{k\rightarrow +\infty \atop l\rightarrow +\infty}\mathbb{E}\big[e^{-\lambda\sum_{n\geq 1}1_{\{U^{l,k}_n\leq u, Z^{l,k}_n>z\}}}\big],\eeqnn
and by (\ref{two limits}), 
\begin{align*}
N_{\Psi}(X_{s_l}>\epsilon_k; Z>z)&=\ell_{s_l}((\epsilon_k,+\infty))\mathbb{P}(Z_n^{l,k}>z)\\
&=\ell_{s_l}((\epsilon_k,+\infty))\mathbb{P}(\bar{Z}_{V^{(n),l,k}_0}>e^{s_l}z)\\
&=\int_{(\epsilon_k,\infty]}\ell_{s}(\ddr x)\left(1-\mathbb{P}(\bar{Z}_x<e^{s_l}z)\right)\\
&=\int_{(\epsilon_k,+\infty]}\ell_{s_l}(\ddr x)\left(1-e^{-xG^{-1}(ze^{s_l})}\right)\\
&\underset{k \rightarrow +\infty}{\longrightarrow}  \int_{(0,+\infty]}\ell_{s_l}(\ddr x)\left(1-e^{-xG^{-1}(ze^{s_l})}\right)
=v_{s_l}(G^{-1}(ze^{s_l})).
\end{align*}
For any $\lambda$, $v_{s_l}(\lambda)\underset{l\rightarrow +\infty}{\longrightarrow} \lambda$  and $G^{-1}(e^{s_l}z)\underset{l\rightarrow +\infty}{\longrightarrow} G^{-1}(z)$. Thus, $N_{\Psi}(Z>z)=G^{-1}(z)$ and $\mathcal{M}=\sum_{i\in I}\delta_{(x_i,Z_{i})}$ is a Poisson point process with intensity $\ddr x\otimes \mu(\ddr z)$.
\end{proof}

\begin{lemma}\label{thethreefinitevar}
Assume $\Psi'(0)=-\infty$, $\int_{0}\frac{\ddr u}{\Psi(u)}=-\infty$ and $\Psi$ is of finite variation. Consider a flow of CSBPs $(X_t(x), t\geq 0, x\geq 0)$ as in (\ref{poissonfinitevariation}), then almost-surely for all $i\in I$, 
\begin{equation}\label{threelimitsfinitevar}
Z_i:=\underset{t\rightarrow +\infty}{\lim}e^{-t}G\left(\frac{1}{X^i_{t-t_i}}\wedge \rho \right).
\end{equation}
The point process 
$\mathcal{M}:=\sum_{i\in I}\delta_{(x_i, Z_i)}$  is a Poisson point process with intensity $\ddr x\otimes \mu(\ddr z)$ where $\bar{\mu}(z)=G^{-1}(z)$. 
\end{lemma}
\begin{proof} Consider the flow of CSBP  $(X_t(x), t\geq 0, x\geq0)$ defined by (\ref{poissonfinitevariation}). Let $(s_l, l\geq 1), (\epsilon_k, k\geq 1)$ be two sequences of positive real numbers such that $s_l\uparrow\infty$ as $l\rightarrow\infty$ and $\epsilon_k\downarrow0$ as
$k\rightarrow\infty$. For any fixed $l$ and $k$, define $I_{l,k}=\{i \in I, t_i\leq s_l, X^{i}_{t_i}> \epsilon_k\}.$ Note that 
 \beqnn
 \sum_{i\in I_{l,k}}\delta_{(x_i, t_i, X^i_{\cdot})}=\sum_{n=1}^\infty\delta_{(U_n^{l,k},T_n^{l,k},V_{\cdot}^{(n),l,k})},
 \eeqnn
where $(U^{l,k}_n, n\geq 1)$ are the arrival times of the Poisson process $$(N_{l,k}(x), x\geq 0):=(\#\{x_i\leq x, t_i\leq s_l, X^i_{t_i}>\epsilon_k\}, x\geq 0)$$ whose parameter 
 $C_{l,k}:=\textbf{d}^{-1}(1-e^{-\textbf{d}s_l})\pi((\epsilon_k,+\infty))$ if $\textbf{d}\neq 0$,  and $C_{l,k}:=s_l\pi((\epsilon_k,+\infty))$ if $\textbf{d}=0$.
The random variables $(T_n^{l,k})_{n\geq1}$ form a sequence of i.i.d.\,random variables with law given by $(1-e^{-\textbf{d}s_l})^{-1}\textbf{d}e^{-\textbf{d}t}1_{[0,s_l]}(t)\ddr t$ if $\textbf{d}\neq 0$, and by $s_l^{-1}1_{[0,s_l]}(t)dt$ if $\textbf{d}=0$. The processes $(V^{(n),l,k}_{\cdot})_{n\geq 1}$ form a sequence of i.i.d.\;CSBPs$(\Psi)$ with initial 
value $V^{(n),l,k}_0$ whose law is $\pi((\epsilon_k,+\infty))^{-1}1_{(\epsilon_k,\infty)}(r)\pi(\ddr r)$. Moreover $(U^{l,k}_n)_{n\geq 1}$, $(T_n^{l,k})_{n\geq1}$ and $(V^{(n),l,k}_\cdot)_{n\geq 1}$
are independent of each other. It follows from Lemma \ref{lemma} that for all $i\in I_{l,k}$, almost-surely, 
\beqlb\label{two limits 2}
e^{-t}G\left(\frac{1}{X^i_t}\wedge \rho\right)=e^{-t_i}e^{-(t-t_i)}G\left(\frac{1}{X^i_{t_i+(t-t_i)}}\wedge \rho\right)\underset{t\rightarrow +\infty}{\longrightarrow} Z_i,
\eeqlb where $\{(x_i, Z_i), i\in I_{l,k}\}=\{(U^{l,k}_n, Z^{l,k}_n), n=1,2,\ldots\}$ and $(Z^{l,k}_n)_{n\geq 1}$ is a sequence of i.i.d.\;non-negative random variables independent of the sequence $(U^{l,k}_n,n\geq 1)$.
For each fixed $n$, $Z^{l,k}_n$ has the same distribution as $e^{-T_n^{l,k}}\bar{Z}({V_0^{(n),l,k}})$
where $(\bar{Z}(x), x\geq0)$ is an extremal process independent of $T_n^{l,k}$ and $V_0^{(n),l,k}$. As in the infinite variation case, we deduce the existence of an almost-sure event $\Omega_0$ on which the limits  in (\ref{two limits 2}) exist for all $i\in I$ almost-surely. Then on $\Omega_0$ set $\mathcal{M}:=\sum_{i\in I}\delta_{(x_i, Z_i)}$. It is a Poisson point process whose intensity $\ddr x\otimes \mu(\ddr z)$ verifies
\begin{align*}
\bar{\mu}(z)&=\displaystyle\lim_{l\rightarrow\infty \atop k\rightarrow \infty}C_{l,k}\mathbb{P}(e^{-T_n^{l,k}}Z_{V_0^{(n),l,k}}>z)\\ 
&=\int_{0}^{+\infty}e^{-\textbf{d}t}\ddr t\int_{0}^{+\infty}\pi(\ddr r)\mathbb{P}(e^{-t}Z_r>z)\\
&=\int_{0}^{+\infty}e^{-\textbf{d}t}\ddr t\int_{0}^{+\infty}\pi(\ddr r)\left(1-e^{-rG^{-1}(ze^{t})}\right) \text{ by Theorem \ref{main11}}\\
&=\int_{0}^{+\infty}e^{-\textbf{d}t}\ddr t \int_{0}^{+\infty}\pi(\ddr r)\left(1-e^{-rv_{-t-\log(z)}(\lambda_0)}\right)\text{by Lemma \ref{slowvariation}: } G^{-1}(ze^{t})=v_{\log(e^{-t}/z)}(\lambda_0)\\
&=\int_{0}^{+\infty}e^{-\textbf{d}t}\ddr t\left(-\Psi(v_{-t-\log(z)}(\lambda_0))+\textbf{d}v_{-t-\log(z)}(\lambda_0)\right) \text{ since } \Psi \text{ has the form (\ref{finitevariationpsi}}).
\end{align*} From the last equality above, we see that
\begin{align*}
\bar{\mu}(z)&=\int_{0}^{\log(1/z)}e^{-\textbf{d}t}\ddr t\left(-\frac{\ddr v_{-t+\log(1/z)}(\lambda_0))}{\ddr t}+\textbf{d}v_{-t+\log(1/z)}(\lambda_0)\right) \text{ since } v_{t}(\lambda) \text{ satisfies (\ref{cumulantintegral})}
\\
&\qquad +\int_{\log(1/z)}^{+\infty}e^{-\textbf{d}t}\ddr t\left(-\frac{\ddr v_{-t}(v_{\log(1/z)}(\lambda_0))}{\ddr t}+\textbf{d}v_{-t}(v_{\log(1/z)}(\lambda_0))\right) \text{ since } v_{-t}(\lambda) \text{ satisfies (\ref{v-})}
\\
&=\left[-e^{-\textbf{d}t}v_{-t+\log(1/z)}(\lambda_0))\right]^{t=\log(1/z)}_{t=0}+\left[-e^{-\textbf{d}t}v_{-t}(v_{\log(1/z)}(\lambda_0))\right]^{t=+\infty}_{t=\log(1/z)} \text{ by integration by parts.}
\end{align*}
We show that $\Psi'(0+)=-\infty$ entails $v_{-t}(\lambda)e^{-\textbf{d}t}\underset{t\rightarrow +\infty}{\longrightarrow} 0$ for any $\lambda$. Since $\Psi'(0+)=-\infty$, for any $b>0$, there exists $\lambda_1$ such that for all $u\leq \lambda_1$, $|\Psi(u)|\geq bu$. Therefore by applying (\ref{v-}), one has
$$t\leq \int_{\lambda}^{\lambda_1}\frac{\ddr u}{|\Psi(u)|}+\int_{v_{-t}(\lambda)}^{\lambda_1}\frac{\ddr u}{|\Psi(u)|}$$ and then for any $b>0$
$$v_{-t}(\lambda)\leq c_{\lambda_1}e^{-bt}$$ for a certain constant $c_{\lambda_1}$. Therefore  $v_{-t}(\lambda)e^{-\textbf{d}t}\underset{t\rightarrow +\infty}{\longrightarrow} 0$ and it comes
$$\bar{\mu}(z)=-e^{-\textbf{d}\log(1/z)}\lambda_0+v_{\log(1/z)}(\lambda_0)-0+e^{-\textbf{d}\log(1/z)}\lambda_0=G^{-1}(z).$$
\end{proof}
\begin{lemma}\label{Zsup} 
For any fixed $x\geq 0$, almost-surely 
$$\forall y\geq x, \underset{t\rightarrow +\infty}{\lim}e^{-t}G\left(\frac{1}{X_t(y)-X_t(x)}\wedge \rho\right)=Z(x,y):=\sup_{x< x_i\leq y}Z_i.$$
\end{lemma}
\begin{proof}
Fix $x\in\mathbb{R}_+$, $\big(X_t(y)-X_t(x): t\geq0, y\geq x\big)$ is a flow of CSBPs$(\Psi)$ independent of $(X_t(y), t\geq0, 0\leq y\leq x)$. Then 
without loss of generality we only consider the case $x=0$. It follows from Theorem \ref{main11} that for any $y\geq0$, almost surely $\underset{t\rightarrow +\infty}{\lim}e^{-t}G\left(\frac{1}{X_t(y)}\wedge \rho\right)=\tilde{Z}(y)$. Furthermore, by  Lemma \ref{thethreeinfinitevar} and Lemma \ref{thethreefinitevar}, we have
 some $\Omega_0$ with $\mathbb{P}(\Omega_0)=1$ such that on $\Omega_0$ for all $i\in I$ $\underset{t\rightarrow +\infty}\lim e^{-t}G\left(\frac{1}{X^i_t}\wedge \rho\right)=Z_i$. Note that $G$ is non-increasing and for all $i\in I$ with $x_i\leq y$, $X_t(y)=\underset{0<x_j\leq y}\sum X^{j}_t\geq X^{i}_t$. Therefore for fixed $y$, almost surely $\tilde{Z}(y)\geq Z_i$ for all $i$ such that $0<x_i\leq y$. This entails that for any fixed $y$, $$\tilde{Z}(y)\geq \underset{0<x_i\leq y}{\sup}Z_i=:Z(0,y), \ a.s.$$ 
Since $\tilde{Z}(y)$ and $Z(0,y)$ have the same law, we can conclude that for any fixed $y$,  $\tilde{Z}(y)=Z(0,y)$ a.s. Then we have some $\Omega_1$ with $P(\Omega_1)=1$ such that on $\Omega_1$ for all $q\in \mathbb{Q}_+$, $$\underset{t\rightarrow +\infty}\lim e^{-t}G\left(\frac{1}{X_t(q)}\wedge \rho\right)=\tilde{Z}(q)=Z(0,q).$$
We now work deterministically on $\Omega_0\cap\Omega_1$. For all $y$ and $q\in \mathbb{Q}$, such that $y<q$, and any $x_i\leq y$, 
$$e^{-t}G\left(\frac{1}{X^i_t}\wedge \rho\right)
\leq e^{-t}G\left(\frac{1}{X_t(y)}\wedge \rho\right)\leq e^{-t}G\left(\frac{1}{X_t(q)}\wedge \rho\right).$$
Then for all $i$ with $x_i\leq y$ 
$$Z_i\leq \underset{t\rightarrow +\infty}{\liminf}\ e^{-t}G\left(\frac{1}{X_t(y)}\wedge \rho\right)\leq \underset{t\rightarrow +\infty}{\limsup}\ e^{-t}G\left(\frac{1}{X_t(y)}\wedge \rho\right)\leq Z(0,q).$$
Therefore
$$Z(0,y)=\sup_{0<x_i\leq y}Z_i\leq \underset{t\rightarrow +\infty}{\liminf}\ e^{-t}G\left(\frac{1}{X_t(y)}\wedge \rho\right)\leq \underset{t\rightarrow +\infty}{\limsup}\ e^{-t}G\left(\frac{1}{X_t(y)}\wedge \rho\right)\leq Z(0,q).$$
Proposition 4.7-ii) in \cite{Res87} ensures that the process $(Z(0,y), y\geq0)$ is c\`adl\`ag, therefore by letting $q$ to $y$, we obtain $\underset{t\rightarrow +\infty}{\lim}e^{-t}G\left(\frac{1}{X_t(y)}\wedge \rho\right)=Z(0,y)$ for all $y\geq 0$ almost-surely. 
\end{proof}
We write $Z(x)$ for $Z(0,x)$. The first statement of Theorem \ref{main1}-i) is now established. It remains to prove that the super-prolific individuals  correspond to the jumps of $(Z(x), x\geq 0)$. 
\begin{lemma}
$\mathcal{P}=\{x_i; Z_i>0, \forall i\in I\}$ a.s.
\end{lemma}
\begin{proof}
If $Z_i>0$ then $X^i_t\underset{t\rightarrow +\infty}{\longrightarrow} +\infty$ and $x_i$ is prolific. One can check that in both infinite variation and finite variation cases, the Poisson point process $\sum_{i\in I}1_{\{Z_i>0\}}\delta_{x_i}$ has intensity $\rho\ddr x$, therefore
$\mathcal{P}=\{x_i, i\in I; Z_i>0\}.$
\end{proof}
\begin{lemma}\label{left limit convergence 1} Under the conditions of Lemma \ref{thethreeinfinitevar}, we have that 
 almost-surely for all $i\in I$, \beqnn
   \underset{t\rightarrow +\infty}{\lim}e^{-t}G\left(\frac{1}{X_t(x_i-)}\wedge \rho\right)=Z(x_i-)=\sup_{0<x_j<x_i}Z_j.
\eeqnn
 \end{lemma}
\begin{proof}
Recall $I_{l,k}=\{i \in I, X^{i}_{s_{l}}\geq \epsilon_k\}.$ As in Lemma \ref{thethreeinfinitevar}, it suffices to show that 
for any fixed $k,l>0$,
\beqlb\label{left limit convergence 1 lk}
\mathbb{P}\Big(\text{ for all } i\in I_{l,k},\lim_{t\rightarrow+\infty}
e^{-t}G\Big(\frac{1}{X_t(x_i-)}\wedge\rho\Big)
=Z(x_i-)\Big)=1.
\eeqlb
Observe that \begin{align}\label{3.32}
\{(x_i,(X_{s_{l}+t}(x_i-))_{t\geq 0}), i\in I_{l,k}\}&=\{(U^{l,k}_n,(X_{s_{l}+t}(U^{l,k}_n-))_{t\geq 0}), n=1,2,\ldots\},
\end{align}
where $(U^{l,k}_n, n\geq 1)$ are the arrival times of the Poisson process given by  (\ref{poisson process 1})
and for fixed $n$,
$(X_{s_{l}+t}(U^{l,k}_n-))_{t\geq 0}=\left(\sum_{0<x_i<U^{l,k}_{n}}X^{i}_{s_{l}+t}\right)_{t\geq 0}$ is a CSBP$(\Psi)$ with initial value $X_{s_{l}}(U^{l,k}_{n}-)$. For simplicity we write $U_n$ for $U_n^{l,k}$. It follows from Lemma \ref{lemma} that for each $n$, 
\begin{equation}\label{3.33}
e^{-t}G\left(\frac{1}{X_t(U_n-)}\wedge \rho\right)=e^{-s_l}e^{-(t-s_l)}G\left(\frac{1}{X_{s_l+(t-s_l)}(U_n-)}\wedge \rho\right)\underset{t\rightarrow +\infty}\longrightarrow \hat{Z}_n\ a.s. \\
\end{equation}
By (\ref{3.32})-(\ref{3.33}), to establish (\ref{left limit convergence 1 lk}) we only need to prove that for any fixed $n$, $\hat{Z}_n=Z(U_n-)$ a.s.

Step 1. We claim that for any fixed $n$, $\hat{Z}_n\geq Z(U_n-)$ a.s. In fact, note that $X_{t}(U_n-)=\sum_{x_j<U_n}X^j_{t}$. Since $G$ is a non-increasing function, we have for all $t\geq 0$ and all $i\in I$ such that $x_i<U_n$, $e^{-t}G\left(\frac{1}{X_t(U_n-)}\right)\geq e^{-t}G\left(\frac{1}{X^{i}_t}\right)$. Then the claim follows from (\ref{3.33}) and Lemma \ref{thethreeinfinitevar}.

Step 2.  We use the coupling method to prove that for any fixed $n$, $\hat{Z}_n$ and $Z(U_n-)$ have the same distributions. 
On an extended probability space, let $(Y_t(y), t\geq0, y\geq0)$ be an independent 
copy of $(X_t(x), t\geq0, x\geq0)$ given by 
\beqnn
Y_t(y)=\sum_{y_j\leq y} Y_t^j,\quad t>0, \ y\geq0,
\eeqnn
with $Y_0(y)=y$, where $\mathcal{N}^Y=\sum_{i\in J}\delta_{(y_j, Y^j)}$ is a Poisson point process over $\mathbb{R}_+\times \mathcal{D}$ 
with intensity $\ddr y \otimes N_{\Psi}(\ddr Y)$.  For $s_l>0$ and for all $i\in I$, define
\beqlb\label{Xbar}
\bar{X}^i_t=\sum_{X_{s_l}(x_i-)<y_j\leq X_{s_l}(x_i)}Y_t^j,\  t>0\quad \text{and}\quad\bar{X}^i_0=\Delta X_{s_l}(x_i).
\eeqlb
Note that $X^i_{s_l}=\Delta X_{s_l}(x_i)$ and 
\beqlb\label{coupling}
\sum_{i\in I}1_{\{X^i_{s_l}>0\}}\delta_{(x_i, X^i_{s_l+\cdot})}\overset{d}{=}\sum_{i\in I} 1_{\{\bar{X}^i_0>0\}}\delta_{(x_i, \bar{X}^i_{\cdot})}.\eeqlb
Set $\bar{X}_{t}(x):=\sum_{x_i\leq x}\bar{X}^i_t$, one has
\beqlb\label{Xbar2}
(\bar{X}_{t}(x), \ t\geq0, x\geq0)\overset{d}{=}(X_{s_l+t}(x),\  t\geq0, x\geq0).
\eeqlb
Applying Lemma \ref{thethreeinfinitevar} and Lemma \ref{Zsup} to  $(Y_t(y),y\geq 0, t\geq 0)$, we have that almost surely,  for all $j\in J$ and for all $y\geq0$,
\beqlb\label{two limits for Y}
Z^Y_j:=\underset{t\rightarrow +\infty}{\lim}e^{-t}G\left(\frac{1}{Y^j_t}\wedge \rho\right)\text{ exists and }
\sup_{y_j\leq y}Z_j^Y=\underset{t\rightarrow +\infty}{\lim}e^{-t}G\left(\frac{1}{Y_t(y)}\wedge \rho\right).\eeqlb
By (\ref{Xbar}), $\bar{X}_t(U_n-)=\sum_{x_i<U_n}\bar{X}^i_t=Y_t(X_{s_l}(U_n-))$ and then
 \beqlb\label{coupling limit 1}
 \underset{t\rightarrow +\infty}{\lim}e^{-(s_l+t)}G\left(\frac{1}{\bar{X}_t(U_n-)}\wedge \rho\right)=e^{-s_l}\sup_{y_j\leq X_{s_l}(U_n-)}Z^Y_j \quad a.s.
 \eeqlb
By (\ref{3.33}), (\ref{Xbar2}) $ \underset{t\rightarrow +\infty}{\lim}e^{-(s_l+t)}G\left(\frac{1}{\bar{X}_t(U_n-)}\wedge \rho\right)\overset{d}{=}\hat{Z}_n$. Therefore by (\ref{coupling limit 1}), \begin{equation}\label{Z-}\hat{Z}_n\overset{d}{=}e^{-s_l}\sup_{y_j\leq X_{s_l}(U_n-)}Z^Y_j.
\end{equation}
By definition of $(\bar{X}^i_t)_{t\geq 0}$, for $i\in I$ such that $\Delta X_{s_l}(x_i)>0$, $\bar{X}^i_t=Y_t(X_{s_l}(x_i))-Y_t(X_{s_l}(x_i-))$. If $\Delta X_{s_l}(x_i)=0$, then we set $\bar{X}^i_t=0$ for all $t\geq 0$. By Lemma \ref{lemma} for any $i\in I$, 
 \beqlb\label{coupling limit 2}
 \bar{Z}_i:=\underset{t\rightarrow +\infty}{\lim}e^{-(s_l+t)}G\left(\frac{1}{\bar{X}^i_t}\wedge \rho\right) \text{ exists a.s. and } \bar{Z}_i\overset{d}{=}e^{-s_l}\sup_{y_j\in(X_{s_l}(x_i-),\;X_{s_l}(x_i)]} Z^Y_j.  
 \eeqlb
Fix $i$. By (\ref{Xbar}), $\bar{X}^i_t\geq Y_t^j$ for any $j\in J$ such that $y_j\in(X_{s_l}(x_i-),\;X_{s_l}(x_i)]$. It follows from the above limit
and (\ref{two limits for Y}) that $\bar{Z}_i\geq e^{-s_l}\sup_{y_j\in(X_{s_l}(x_i-),\;X_{s_l}(x_i)]} Z^Y_j$ a.s. Thus for each $i\in I$, $\bar{Z}_i= e^{-s_l}\sup_{y_j\in(X_{s_l}(x_i-),\;X_{s_l}(x_i)]} Z^Y_j$ a.s. and
\beqnn
e^{-s_l}\sup_{y_j\leq X_{s_l}(U_n-)}Z^Y_j=e^{-s_l}\sup_{x_i<U_n}\sup_{y_j\in(X_{s_l}(x_i-),\;X_{s_l}(x_i)]} Z^Y_j=\sup_{x_i<U_n}\bar{Z}_i, \ a.s.
\eeqnn
By (\ref{coupling}), $\sup_{x_i<U_n}\bar{Z}_i\overset{d}{=}\sup_{x_i<U_n}Z_i=Z(U_n-)$. 
Therefore, (\ref{Z-}) entails that for any $n$, $\hat{Z}_n\overset{d}{=}Z(U_n-)$.
\end{proof}

\begin{lemma}\label{left limit convergence 2} Under the conditions of Lemma \ref{thethreefinitevar}, we have that 
 almost-surely for all $i\in I$, \beqnn
   \underset{t\rightarrow +\infty}{\lim}e^{-t}G\left(\frac{1}{X_t(x_i-)}\wedge \rho\right)=Z(x_i-)=\sup_{0<x_j<x_i}Z_j.
\eeqnn
 \end{lemma}

\begin{proof}  Recall $I_{l,k}:=\{i \in I, t_i\leq s_l, X^{i}_{t_i}> \epsilon_k\}$, $U^{l,k}_n$ and $T_n^{l,k}$ defined in the proof of Lemma \ref{thethreefinitevar}.  For simplicity 
we write $U_n$ and $T_n$ for  $U^{l,k}_n$ and $T_n^{l,k}$. Observe that
\beqnn
\{(x_i,t_i, (X_{t_i+t}(x_i-))_{t\geq 0}), i\in I_{l,k}\}=\{(U_n, T_n, (X_{T_n+t}(U_n-))_{t\geq 0}), n=1,2,\ldots\},
\eeqnn
 where for each $n$, 
\beqlb\label{representation in finite variation}
X_{T_n+t}(U_n-)=e^{-\textbf{d}(t+T_n)}U_n+\sum_{0<x_j<U_n}X^{j}_{t+T_n-t_j}1_{\{t_j<t+T_n\}}
\eeqlb
 is a CSBP$(\Psi)$ with initial value $X_{T_n}(U_{n}-)$.  It follows from Lemma \ref{lemma} that for each $n$,  
 $
e^{-t}G\left(\frac{1}{X_t(U_n-)}\wedge \rho\right)\underset{t\rightarrow +\infty}\longrightarrow \hat{Z}_n\ a.s.
$
As proved in step 1 of Lemma \ref{left limit convergence 1},  it follows from (\ref{representation in finite variation}) and (\ref{threelimitsfinitevar}) that $\hat{Z}_n\geq\sup_{x_i<U_n}Z_i=Z(U_n-)$ a.s. Then we use the coupling method to prove that $\hat{Z}_n$ and $Z(U_n-)$ have the same distributions.
On an extended probability space, let $(Y_t(y), t\geq0, y\geq0)$ be an independent 
copy of $(X_t(x), t\geq0, x\geq0)$ given by 
 \beqnn
  Y_{t}(y)=e^{-\textbf{d}t}y+\sum_{y_{j}\leq y}1_{\{s_j\leq t\}}Y^{j}_{t-s_j}, \ t\geq0,  \ y\geq0,
\eeqnn
where $\mathcal{N}^Y=\sum_{j\in J}\delta_{(y_j, s_j, Y^j)}$ is a Poisson point process over $\mathbb{R}_+\times \mathbb{R}_+\times \mathcal{D}$ with intensity $\ddr y \otimes e^{-\textbf{d}s}\ddr s\otimes \int_{0}^{+\infty}\pi(\ddr r)\mathbb{P}_{r}^{\Psi}(\ddr Y)$.  For $i\in I$ with $t_i\leq T_n$, set
\beqnn
\bar{X}^i_t=X^i_t,\quad\text{if}\ 0\leq t\leq T_n-t_i;
\eeqnn
\beqnn
\bar{X}^i_t= e^{-\textbf{d}(t-T_n+t_i)}\Delta X_{T_n}(x_i)
+\sum_{j\in J_i}1_{\{s_j\leq t-T_n+t_i\}}Y^{j}_{t-T_n+t_i-s_j},\quad \text{if}\ t>T_n-t_i,
\eeqnn
where $J_i=\{j\in J: X_{T_n}(x_i-)<y_j\leq X_{T_n}(x_i)\}$. For $i\in I$ with $t_i>T_n$, set $\bar{X}^i_t=X^i_t$ for $t\geq0$. One has
\beqnn
\sum_{i\in I}\delta_{(x_i, t_i, X^i_{\cdot})}\overset{d}{=} \sum_{i\in I}\delta_{(x_i, t_i, \bar{X}^i_{\cdot})}.\eeqnn
The flow $(\bar{X}_{t}(x), t\geq 0, x\geq 0)=(e^{-\textbf{d}t}x+\sum_{x_{i}\leq x}1_{\{t_i\leq t\}}\bar{X}^{i}_{t-t_i}, t\geq 0, x\geq 0)$ is a flow of CSBPs$(\Psi)$ and one can show that $\hat{Z}_n\overset{d}{=} Z(U_n-)$ as in Step 2 of Lemma  \ref{left limit convergence 1}. \end{proof}
\begin{lemma}\label{superprolific} $\mathcal{S}\cap \mathcal{P}=\{x>0; \Delta Z(x)>0\}$ a.s.
\end{lemma}
\begin{proof} We focus on the infinite variation case, the proof in the finite variation case is similar by replacing $X^i_t$ by $X^{i}_{t-t_i}$. We first establish that if
$x_j$ and $x_m$ are two individuals such that $Z_j>Z_m$, then $\frac{X^{j}_t}{X^{m}_t}\underset{t\rightarrow +\infty}\longrightarrow +\infty$ a.s. If $Z_m=0$, then clearly $\frac{X^{j}_t}{X^{m}_t}\underset{t\rightarrow +\infty}\longrightarrow +\infty$, as $X^{m}_t \underset{t\rightarrow +\infty}\longrightarrow 0$ and $X^{j}_t \underset{t\rightarrow +\infty}\longrightarrow +\infty$. Assume now $Z_m>0$, 
Recall $G(z)=\exp\left(-\int_{z}^{\lambda_0}\frac{\ddr u}{\Psi(u)}\right)$, by Lemma \ref{thethreeinfinitevar}, 
$$\frac{e^{-t}G(1/X^m_t)}{e^{-t}G(1/X^j_t)}=\exp\left(-\int_{\frac{1}{X_t^m}}^{\frac{1}{X_t^j}}\frac{\ddr u}{\Psi(u)}\right)\underset{t\rightarrow +\infty}\longrightarrow\frac{Z_m}{Z_j}<1 \text{ a.s.}$$
Then, $\underset{t\rightarrow +\infty}{\lim}\int_{\frac{1}{X_t^m}}^{\frac{1}{X_t^j}}\frac{\ddr u}{\Psi(u)}\in (0,+\infty)$. Since $\Psi$ is non-positive in a neighbourhood of $0$, then for $t$ large enough $\frac{1}{X_t^{j}}\leq \frac{1}{X_t^{m}}$. This ensures that $\underset{t\rightarrow +\infty}{\liminf} \frac{X^j_t}{X^m_t}>0$. Assume $\underset{t\rightarrow +\infty}{\liminf} \frac{X^j_t}{X^m_t}=:\theta\in (0,+\infty)$, then there exists some sequence $(t_n)_{n\geq1}$ such that $t_n\underset{n\rightarrow +\infty}\longrightarrow +\infty$ and  $\lim_{n\rightarrow\infty}
\frac{X^j_{t_n}}{X^m_{t_n}}=\theta.$ Since $G$ is slowly varying at $0$, by locally uniform convergence (see Proposition 0.5 in \cite{Res87}) we have that 
\beqnn
\frac{G(\frac{1}{X^m_{t_n}})}{G(\frac{1}{X^j_{t_n}})}=\frac{G(\frac{1}{X^j_{t_n}}\frac{X^j_{t_n}}{X^m_{t_n}})}{G(\frac{1}{X^j_{t_n}})}\underset{n\rightarrow +\infty }{\longrightarrow} 1 \ a.s.
\eeqnn
This entails a contradiction. In conclusion, $\theta=+\infty$ and $\frac{X^{j}_t}{X^{m}_t}\underset{t\rightarrow +\infty}\longrightarrow +\infty$ a.s.  We show now $\mathcal{S}\cap \mathcal{P}=\{x>0; \Delta Z(x)>0\}$. 
\begin{itemize}
\item[-] If $x=x_{m}$ is not a jump of $(Z(x), x\geq 0)$, then $Z_m<\sup_{x_i\leq x_m}Z_i$. Therefore there exists $j$ such that $x_j<x_m$ and $Z_j>Z_m$. Since $\frac{X_t(x_m-)}{X_t^{m}}\geq \frac{X^{j}_t}{X^{m}_t}$, we get $\underset{t\rightarrow +\infty}{\lim}\frac{X_t^{m}}{X_t(x_m-)}=0$, that is to say $x_m$ is not a super-individual and $\mathcal{S}\cap \mathcal{P}\subset \{x>0: \Delta Z(x)>0\}$. 
\item[-] If $x=x_m$ is a jump of $(Z(x), x\geq 0)$, then by Lemma \ref{Zsup}, $Z_m=\sup_{x_i\leq x_m}Z_i$ and $Z_m>0$. By Lemma \ref{left limit convergence 1}, $e^{-t}G\left(1/X_t(x_m-)\wedge \rho\right)\underset{t\rightarrow +\infty}\longrightarrow Z(x_m-)<Z_m$. Therefore, 
$$\frac{G\left(1/X_t(x_m-)\wedge \rho\right)}{G(1/X_t^m)}\underset{t\rightarrow +\infty}\longrightarrow \frac{Z(x_m-)}{Z_m}<1.$$
Using the slow variation property of $G$ as before, we show that $$\underset{t\rightarrow +\infty}{\lim}\frac{X^m_t}{X_t(x_m-)}=+\infty,$$
thus $x_m$ is a super-individual.
\end{itemize}
\end{proof}
By combining Lemma \ref{Zsup} and Lemma \ref{superprolific}, one obtains Theorem \ref{main1}-i).

\section{Subcritical processes}\label{subcritical}
Consider a subcritical or critical CSBP $(X_{t}(x), t\geq 0)$. Since it goes to $0$ almost-surely,  there is no prolific individuals and $\Delta X_{t}(x)$ goes to $0$ with probability $1$ for all $x\geq 0$. If not otherwise specified, we consider $\Psi$ such that $\Psi'(0+)\geq 0$ and $\int^{+\infty}\frac{\ddr u}{\Psi(u)}=+\infty$. As recalled in Theorem \ref{Grey}, the CSBP($\Psi$) goes to zero without being absorbed. We mention that there is no notion of persistence in the framework of discrete state-space branching processes, so that most results in this section have no discrete counterparts. 
Recall the set of super-individuals $$\mathcal{S}:=\left\{x>0; \underset{t\rightarrow +\infty}\lim \frac{\Delta X_{t}(x)}{X_{t}(x-)}=+\infty\right\}.$$
We start with the case of a branching mechanism with finite variation in which no super-individuals exist.
\begin{Prop}[Proposition 2.3 and Lemma 2.4 in \cite{DuqLab}]\label{limit1}\;
Suppose $\textbf{d}\in \mathbb{R}$ and fix $\lambda\in (0,+\infty)$. Consider a flow of CSBPs($\Psi$) $(X_t(x),x\geq 0,t\geq 0)$ defined as in (\ref{poissonfinitevariation}). There exists a c\`adl\`ag subordinator $(V^{\lambda}(x), x\geq 0)$ with Laplace exponent $\theta\mapsto v_{-\frac{\log \theta }{\textbf{d}}}(\lambda)$ such that almost-surely for any $x>0$, $$v_{-t}(\lambda)X_{t}(x) \underset{t\rightarrow +\infty}\longrightarrow V^{\lambda}(x) \text{ and } v_{-t}(\lambda)X_{t}(x-) \underset{t\rightarrow +\infty}\longrightarrow V^{\lambda}(x-).$$ 
Moreover, $(V^{\lambda}(x), x\geq 0)$ has an infinite L\'evy measure and $\mathcal{S}=\emptyset$  almost-surely.
\end{Prop}
\begin{proof}
We only verify that $\mathcal{S}=\emptyset$. Lemma 2.4 in \cite{DuqLab} entails that for all $x>0$, $\underset{t\rightarrow +\infty}{\lim} \frac{\Delta X_{t}(x)}{X_{t}(x-)}=\frac{\Delta V^{\lambda}(x)}{V^{\lambda}(x-)}<+\infty$. Therefore $\mathcal{S}=\emptyset.$ 
\end{proof}

We study now the subcritical persistent processes with infinite variation.
\begin{theorem}\label{main2}\;
Suppose that $\Psi$ is  (sub)critical, persistent and of infinite variation. Fix some $\lambda_0>0$ and set $G(y)=\exp\left(-\int_{\lambda_0}^{y}\frac{\ddr u}{\Psi(u)}\right)$ for $y\in (0,+\infty)$. Then, for all $x\geq 0$, almost-surely
$$e^{t}G\left(\frac{1}{X_{t}(x)}\right)\underset{t\rightarrow +\infty}\longrightarrow \tilde{Z}(x) \text{ a.s.}$$ 
where $(\tilde{Z}(x), x\geq 0)$ is a positive extremal-$F$ process (in the sense of (\ref{extremalprocess})) with $F(z)=\exp\left(-G^{-1}(z)\right)$ for $z\geq 0$ and $G^{-1}(z)=v_{\log z}(\lambda_0)$ for all $z\geq 0$. 
\end{theorem}
\begin{example} \label{example2} Consider the mechanism $\Psi(u)=(u+1)\log(u+1)$. Fix $\lambda_0=e-1$, then for all $z\geq 0$, $G(z)=\frac{1}{\log(1+z)}$ and $G^{-1}(z)=e^{1/z}-1$. The process $(\tilde{Z}(x), x\geq 0)$ is an extremal-$F$ process with $F(z)=e^{1-e^{1/z}}$ for all $z\geq 0$.
\end{example}

\begin{lemma} If $\textbf{d}=+\infty$ and $\int^{+\infty}\frac{\ddr u}{\Psi(u)}=+\infty$, then the map $G:y\mapsto \exp\left(-\int_{\lambda_0}^{y}\frac{\ddr u}{\Psi(u)}\right)$ is continuous, non-increasing, goes from $[0,+\infty]$ to $[0,+\infty]$ and is slowly varying at $+\infty$. Moreover $G^{-1}(z)=v_{\log z}(\lambda_0)$ for $z\in [0,+\infty]$.
\end{lemma}
\begin{proof} The map $G$ is continuous, non-increasing and maps $(0,+\infty)$ to $(0,+\infty)$. Indeed $\Psi$ is continuous positive on $(0,+\infty)$ and $\lim_{y\rightarrow+\infty}G(y)=0$, since $\int_{\lambda_0}^{+\infty} \frac{1}{\Psi(z)}dz=+\infty$, $\lim_{y\rightarrow 0}G(y)=+\infty$ since 
$\int_0^{\lambda_0} \frac{\ddr z}{\Psi(z)}=+\infty$. Let $\delta>0$, since $\int_0^1u \pi(\ddr u)=+\infty$, 
$\textbf{d}:=\lim_{z\rightarrow+\infty}\Psi(z)/z=+\infty$ and for any $b>0$, $\Psi(z)\geq bz$ for large enough $z$. Therefore
$\left\lvert \int_x^{\delta x}\frac{\ddr z}{\Psi(z)}\right\lvert \leq \frac{|\log(\delta)|}{b}.$ Since $b$ is arbitrarily large, \begin{equation}\label{slowvarinf} \underset{x\rightarrow +\infty}{\lim} \int_x^{\delta x}\frac{\ddr z}{\Psi(z)}=0 
\end{equation}
and the map $G$ is slowly varying at $+\infty$.\end{proof}

\noindent
The proof is similar to that of Theorem \ref{main11}. We provide some details for sake of completeness.

\begin{proof}[Proof of Theorem \ref{main2}].
Under the assumption $\int^{+\infty}\frac{\ddr u}{\Psi(u)}=+\infty$, $\bar{v}=+\infty$ and by applying Lemma \ref{martingale}, we have that for a fixed $\lambda>0$ and any $x\geq0$, $(v_{-t}(\lambda)X_{t}(x), t\geq 0)$ converges $\mathbb{P}_x$-almost-surely,  as $t\rightarrow+\infty$, to some random variable $V^{\lambda}(x)$ taking value in $\bar{\mathbb{R}}_{+}$. Let $\delta>0$, (\ref{equav}) yields 
\beqnn\lim_{t\rightarrow +\infty} \left| \int_{\lambda}^{v_t(\delta v_{-t}(\lambda))}\frac{\ddr z}{\Psi(z)} \right| = \lim_{t\rightarrow +\infty}\left| \int_{v_{-t}(\lambda)}^{\delta v_{-t}(\lambda)}\frac{\ddr  z}{\Psi(z)}\right|.
\eeqnn
Recall that $\lim_{t\rightarrow+\infty}v_{-t}(\lambda)=+\infty$, (\ref{slowvarinf}) implies that $\lim_{t\rightarrow+\infty}v_{t}(\delta v_{-t}(\lambda))=\lambda,$ for any $\delta>0$. We have
\begin{align*}
\lim_{t\rightarrow+\infty}\mathbb{E}[e^{-\delta v_{-t}(\lambda)X_{t}(x)}]&=\lim_{t\rightarrow+\infty}\exp\big(-xv_{t}(\delta v_{-t}(\lambda))\big)
=e^{-x\lambda}.
\end{align*}
The limit does not depend on $\delta$, therefore $(v_{-t}(\lambda)X_{t}(x), t\geq 0)$ converges either to $0$ or $+\infty$ and $\mathbb{P}(V^{\lambda}(x)=0)=e^{-x\lambda}.$ 
For any $x\geq 0$, by considering the collection of random variables $V^{\lambda}(x)$, one can define the random variable $$\Lambda_x:=\inf\{\lambda\in (0,+\infty)\cap \mathbb{Q}: V^{\lambda}(x)=+\infty\}.$$ 
Note that $\mathbb{P}(\exists\;\lambda>0 \ \mbox{such that}\ V_x^\lambda=0)=1$, which implies that $\Lambda_x>0$ a.s.
Similarly as in (\ref{0-infty}), one can show that almost-surely
\beqlb\label{0-infty_sub}
v_{-t}(\lambda)X_{t}(x)\underset{t\rightarrow +\infty}\longrightarrow  \begin{cases} 0 & \text{ if } \Lambda_x>\lambda\\
+\infty &\text{ if }\Lambda_x<\lambda.
\end{cases}
\eeqlb
Choose $\lambda'$ and $\lambda''$ such that $\lambda'<\Lambda_x<\lambda''$. By  (\ref{0-infty_sub}), for large enough $t$, 
 \begin{center}
$v_{-t}(\lambda'')X_{t}(x)\geq 1$ and $0<v_{-t}(\lambda')X_{t}(x)\leq 1$.
\end{center}
Since $G(v_{-t}(\lambda'))=e^{-\int^{\lambda'}_{\lambda_0}\frac{d x}{\Psi(x)}}e^{-t}=G(\lambda')e^{-t}$ then
$$G(\lambda')\geq e^{t}G(1/X_{t}(x))\geq G(\lambda'').$$
Since $\lambda'$ and $\lambda''$ are arbitrarily close to $\Lambda_x$, and $G$ is continuous, we get $$e^{t}G(1/X_{t}(x)) \underset{t\rightarrow +\infty}\longrightarrow G(\Lambda_x)\quad \mathbb{P}\text{-almost surely.}$$
Let $\lambda_1 \neq \lambda_2$ be two positive real numbers, By Lemma \ref{v}, $\lim_{t\rightarrow+\infty}v_{-t}(\lambda_2)/v_{-t}(\lambda_1)=0$ if $\lambda_1>\lambda_2$ and $+\infty$ otherwise. Thus $\lim_{t\rightarrow+\infty} v_{t}(v_{-t}(\lambda_1)+v_{-t}(\lambda_2))= \lambda_1 \vee \lambda_2$. We conclude by applying the same arguments as at the end of the proof of Theorem \ref{main2}. \end{proof}

The following proposition is obtained similarly as Proposition \ref{superexponential1}
\begin{Prop}\label{superexponential2} 
Assume $\textbf{d}=+\infty$ and $\int^{+\infty}\frac{\ddr u}{\Psi(u)}=+\infty$. If there are $\lambda>0$ and $\alpha>0$ such that $\left\lvert\int_{\lambda}^{+\infty}\left(\frac{1}{\Psi(u)}-\frac{1}{\alpha u \log u}\right)\ddr u \right\lvert<+\infty,$
then $G(1/y)\underset{y\rightarrow 0}{\sim} k_{\lambda}\left(\log 1/y\right)^{-1/\alpha}$ for a constant $k_{\lambda}>0$. Fix $x>0$, then \[\log X_t(x)\underset{t\rightarrow +\infty}{\sim} -e^{\alpha t}k_{\lambda}^{\alpha}\tilde{Z}(x)^{-\alpha} \text{ a.s.}\]
\end{Prop}
A natural extremal process associated to a flow of non-persistent subcritical CSBPs is given by the extinction times of each initial individuals. The proof of the following proposition is straightforward.
\begin{Prop}[Theorem 0.3-i) in \cite{DuqLab} for $\zeta=\zeta_0$ ]\label{flownonpersistent}
Assume $\Psi'(0)\geq 0$ and $\int^{+\infty}\frac{\ddr u}{\Psi(u)}<+\infty$, all individuals living at time $0$ get absorbed in $0$. For all $x\geq 0$, set $Z(x):=\inf\{t\geq 0; X_t(x)=0\}$. The process $(Z(x), x\geq 0)$ is an extremal-$F$ process, with $F(z)=\exp\left(-\bar{v}_{z}\right)$. Let $Z_i$ be the life-length of $X^i$. The point process $\mathcal{M}:=\sum_{i\in I}\delta_{(x_i,Z_i)}$ is a Poisson point process with intensity $\ddr x\otimes \mu(\ddr z)$ with $\bar{\mu}(z)=\overline{v}_{z}$. Moreover, almost-surely for all $x\geq 0$, $Z(x)=\sup_{x_{i}\leq x} Z_i$
and
$$\mathcal{S}:=\left\{x>0; \exists t>0; \Delta X_t(x)>0 \text{ and } X_{t}(x-)=0 \right\}=\{x>0; \Delta Z(x)>0\}.$$
\end{Prop}
\begin{example} Let $\alpha\in (0,1]$ and $\Psi(u)=\alpha u^{\alpha+1}$. Consider a flow of CSBPs$(\Psi)$ $(X_t(x), x\geq 0, t\geq 0)$. By Proposition \ref{flownonpersistent} and Theorem \ref{Grey}-ii), the process $(Z(x), x\geq 0)$ is an extremal-$F$ process with $F$ the probability distribution function of a Fr\'echet law with parameter $\frac{1}{\alpha}\in [1,+\infty)$, that is to say $F(z)=e^{-z^{-\frac{1}{\alpha}}}$ for all $z\geq 0$.
\end{example}
\subsection{Proof of Theorem \ref{main1}-ii)}
In the subcritical case with infinite variation, one always has $F(0)=0$. The state $0$ is therefore instantaneous for the process $(Z(x), x\geq 0)$. We work now with $\Psi$ subcritical persistent with infinite variation and consider a flow of CSBPs$(\Psi)$ as in (\ref{poissoninfinitevariation}). The following lemmas can be established by applying \textit{verbatim} the proofs of the supercritical case. 

\begin{lemma}\label{lemmasub} Consider on the same probability space $(X_t, t\geq 0)$ and $(Y_t, t\geq 0)$ two independent CSBPs$(\Psi)$, satisfying the conditions of Theorem \ref{main2} starting from $X_0$ and $Y_0$. Then, almost-surely
$$e^{t}G\left(\frac{1}{X_t}\right)\underset{t\rightarrow +\infty}\longrightarrow \Gamma_1,\ e^{t}G\left(\frac{1}{Y_t}\right)\underset{t\rightarrow +\infty}\longrightarrow \Gamma_2 \text{ and } e^{t}G\left(\frac{1}{X_t+Y_t}\right)\underset{t\rightarrow +\infty}\longrightarrow \Gamma_1\vee \Gamma_2,$$
with $\Gamma_1$ and $\Gamma_2$ independent such that $\Gamma_1\neq \Gamma_2 \text{ a.s.}$ and
$$\Gamma_1\overset{d}{=}\bar{Z}(X_0),\quad  \Gamma_2\overset{d}{=}\bar{Z}(Y_0) \text{ and }\Gamma_1 \vee \Gamma_2 \overset{d}{=}\bar{Z}(X_0+Y_0),$$ where $(\bar{Z}(x), x\geq0)$ is an extremal-$F$ process, independent of $X_0$ and $Y_0$. 
\end{lemma}
\begin{proof} The proof is similar as that of Lemma \ref{lemma}. 
\end{proof}
We gather here in a single lemma the analogues of Lemma \ref{thethreeinfinitevar}, Lemma \ref{Zsup} and Lemma \ref{left limit convergence 1} for the subcritical case. 
\begin{lemma}\label{thethreeinfinitevarsub} 
Almost-surely for all $i\in I$, the following limit exists
\begin{equation}\label{threelimitssub}
Z_i:=\underset{t\rightarrow +\infty}{\lim}e^{t}G\left(\frac{1}{X^i_t}\right).
\end{equation}
The point process 
$\mathcal{M}:=\sum_{i\in I}\delta_{(x_i, Z_i)}$  is a Poisson point process with intensity $\ddr x\otimes \mu(\ddr z)$ where $\bar{\mu}(z)=G^{-1}(z)$. For any $x\geq 0$, almost-surely $$\forall y\geq x, \underset{t\rightarrow +\infty}{\lim} e^{t}G\left(\frac{1}{X_t(y)-X_t(x)}\right)=Z(x,y):=\underset{x<x_i\leq y}{\sup}Z_i.$$
Moreover, almost-surely for all $i\in I$,
$$\underset{t\rightarrow +\infty}{\lim}e^{t}G\left(\frac{1}{X_t(x_i-)}\right)=Z(x_i-).$$
\end{lemma}

\begin{lemma}\label{superresistant} $\mathcal{S}=\{x>0; \Delta Z(x)>0\}$ a.s.  
\end{lemma}
\begin{proof} 
We only show that if $Z_j>Z_m$ then $\frac{X^{j}_t}{X^{m}_t}\underset{t\rightarrow +\infty}\longrightarrow +\infty$ a.s. The identity for $\mathcal{S}$ can be proven as in Lemma \ref{superprolific} by using the slow variation at $+\infty$. Recall $G(z)=\exp\left(-\int_{\lambda_0}^{z}\frac{\ddr u}{\Psi(u)}\right)$, by Lemma \ref{thethreeinfinitevarsub},
$$\frac{e^{t}G(1/X^m_t)}{e^{t}G(1/X^j_t)}=\exp\left(-\int_{\frac{1}{X_t^j}}^{\frac{1}{X_t^m}}\frac{\ddr u}{\Psi(u)}\right)\underset{t\rightarrow +\infty}\longrightarrow\frac{Z_m}{Z_j}<1.$$
Then, $\underset{t\rightarrow +\infty}{\lim}\int_{\frac{1}{X_t^j}}^{\frac{1}{X_t^m}}\frac{\ddr u}{\Psi(u)}\in (0,+\infty)$ a.s. Since $\Psi$ is non-negative then for $t$ large enough $\frac{1}{X_t^{m}}\geq \frac{1}{X_t^{j}}$. This ensures that $\underset{t\rightarrow +\infty}{\liminf} \frac{X^j_t}{X^m_t}>0$ a.s. Assume $\underset{t\rightarrow +\infty}{\liminf} \frac{X^j_t}{X^m_t}=:\theta\in (0,+\infty)$ a.s., then there exists $(t_n)$ such that $t_n\underset{n\rightarrow +\infty}\longrightarrow +\infty$ and $\frac{X^j_{t_n}}{X^m_{t_n}}\underset{n\rightarrow +\infty}{\longrightarrow} \theta$ a.s. By locally uniform convergence of slowly varying function, (see Proposition 0.5 in \cite{Res87}) we have that 
\beqnn
\frac{G(\frac{1}{X^m_{t_n}})}{G(\frac{1}{X^j_{t_n}})}=\frac{G(\frac{1}{X^j_{t_n}}\frac{X^j_{t_n}}{X^m_{t_n}})}{G(\frac{1}{X^j_{t_n}})}\underset{n\rightarrow +\infty }{\longrightarrow} 1 \ a.s.
\eeqnn
This entails a contradiction. In conclusion, if $Z_j>Z_m$, then $\frac{X^{j}_t}{X^{m}_t}\underset{t\rightarrow +\infty}\longrightarrow +\infty$ a.s. \end{proof}
\subsection{Flow of Neveu's continuous-state branching processes}\label{subsectionNeveu}
It is well-known that a supercritical CSBP$(\Psi)$ with $\rho<\infty$ conditioned to be extinct is a subcritical CSBP with mechanism $u \mapsto\Psi(u+\rho)$. When the CSBP is supercritical non-explosive persistent with infinite mean and infinite variation, one can combine Theorems \ref{main11} and \ref{main2} to obtain both the rates of growth and of decay. In the Neveu case, we can renormalize the population size on the events of extinction and non-extinction by the same function. As claimed in the introduction, the renormalized flow of Neveu CSBPs converges towards an extremal-$\Lambda$ process with $\Lambda(z)=e^{-e^{-z}}$ for $z\in \mathbb{R}$.
\begin{lemma}\label{Gumbel} Let $(X_t(x), t\geq 0)$ be a CSBP$(\Psi)$ with $\Psi(u)=u\log u$. Then 
$e^{-t}\log X_t(x)\underset{t\rightarrow +\infty}{\longrightarrow} Z(x)\text{ a.s.}$
where the process $(Z(x),x\geq 0)$ is an extremal-$\Lambda$ process.
\end{lemma}
\begin{proof} 
We refer the reader to Proposition 10 and its proof in Fleischmann and Sturm \cite{MR2086012} for the almost-sure convergence for fixed $x$, towards a random variable with a Gumbel law. We verify now that the process $(Z(x), x\geq 0)$ is extremal. Clearly $Z(x+y)$ has the same law as $Z(x)\vee Z'(y)$ for a random variable $Z'(y)$ distributed as $Z(y)$ and independent of $Z(x)$.  Moreover, since $$e^{-t}\log(X_t(x+y))\geq e^{-t}\log (X_t(x)) \vee e^{-t}\log (X_t(x+y)-X_t(x))$$ then $Z(x+y)\geq Z(x)\vee Z(x,x+y)$ a.s. with $Z(x,x+y):=\underset{t\rightarrow \infty}{\lim} e^{-t}\log (X_t(x+y)-X_t(x))$. We deduce that $Z(x+y)=Z(x)\vee Z(x,x+y)$ a.s. and conclude by recalling (\ref{maxid}).
\end{proof}

If $x$ is such that $Z(x)<0$, then the population started from $x$ is extinguishing. If $x$ is such that $Z(x)>0$, the population is not extinguishing. The process $(Z(x),x\geq 0)$ enters in $(0,+\infty)$ with the first prolific individual. 

\begin{Prop}\label{corneveu} Consider $(X_t(x), t\geq 0, x\geq 0)$ a flow of Neveu CSBPs (constructed as in (\ref{poissoninfinitevariation})).  Then almost-surely for all $i\in I$, the limit $Z_i:=\underset{t\rightarrow +\infty}{\lim}e^{-t}\log X^{i}_t$ exists. The point process  $\mathcal{M}:=\sum_{i\in I}\delta_{(x_i, Z_i)}$ is a Poisson point process over $\mathbb{R}_+\times\mathbb{R}$ with intensity $\ddr x\otimes e^{-z}\ddr z$ and
almost-surely, for all $x\geq 0$, \[e^{-t}\log X_t(x)\underset{t\rightarrow +\infty}{\longrightarrow} Z(x):=\sup_{x_i\leq x}Z_i.\] Moreover $\mathcal{S}=\{x>0; \Delta Z(x)>0\}$ a.s.
\end{Prop}
\begin{proof}  We stress that Proposition \ref{corneveu} is not a direct combination of Theorem \ref{main1}-i) and Theorem \ref{main1}-ii) for $\Psi(u)=u\log u$. Indeed the point process obtained in Theorem \ref{main1}-i) does not take into account the decay of non-prolific individuals (they are all with $Z_i=0$).  With the same notation as in Lemma \ref{thethreeinfinitevar}, for all $i\in I_{l,k}$, by Lemma \ref{Gumbel}, $Z_i$ exists a.s. and has the same law as $e^{-s_l}\bar{Z}_{V^{l,k}_0}$, with $\bar{Z}_x$ a random variable with a Gumbel Law and $V^{l,k}_0$ a random variable with law $\frac{\ell_{s_{l}}(\ddr x; x\geq \epsilon_k)}{\ell_{s_{l}}([\epsilon_k,+\infty))}$. Since the number of individuals in $I_{l,k}$ is a Poisson random variable with parameter  $\ell_{s_{l}}([\epsilon_k,+\infty))$, one has
\begin{align*}
N_{\Psi}(&X_{s_l}>\epsilon_k; Z>z)=\ell_{s_l}((\epsilon_k,+\infty))\mathbb{P}(\bar{Z}_{V^{l,k}_0}>e^{s_l}z)\\
&=\int_{(\epsilon_k,\infty]}\ell_{s}(\ddr x)\left(1-\mathbb{P}(\bar{Z}_x<e^{s_l}z)\right)=\int_{(\epsilon_k,+\infty]}\ell_{s_l}(\ddr x)\left(1-e^{-xe^{-ze^{s_l}}}\right)\\
&\underset{k \rightarrow +\infty}{\longrightarrow}  \int_{(0,+\infty]}\ell_{s_l}(\ddr x)\left(1-e^{-xe^{-ze^{s_l}}}\right)=v_{s_l}(e^{-ze^{s_l}})=e^{-ze^{2s_l}}\underset{l \rightarrow +\infty}{\longrightarrow} e^{-z}.
\end{align*}
Thus, the intensity of the Poisson point process $\mathcal{M}$ is $\mu(\ddr z)=e^{-z}\ddr z$. The rest of the proof follows exactly the same lines as Lemma \ref{thethreeinfinitevar} and Lemma \ref{Zsup}.\end{proof}

\section*{Acknowledgements} We would like to thank the referees for their careful reading and insightful suggestions. This work is partially supported by the visiting scholar program at Chern Institute of Mathematics (CIM), by the French National Research Agency (ANR): ANR GRAAL (ANR-14-CE25-0014) as well as by LABEX MME-DII (ANR11-LBX-0023-01), and by the NSFC of China (11671216). C.F thanks T. Duquesne, C. Labb\'e, V. Rivero and Z. Li for helpful discussions. He would like to  thank J-F. Le Gall for providing Neveu's work \cite{Neveu}. C.M would like to thank Professors A. Lambert and K. Xiang for their encouragements.


\end{document}